\documentclass{article}
\usepackage{graphicx}
\usepackage{amsfonts}
\usepackage{amsmath}
\usepackage{amssymb}
\usepackage{url}
\usepackage{booktabs}
\usepackage{dsfont}
\usepackage{fancyhdr}
\usepackage{indentfirst}
\usepackage{enumerate}
\usepackage[normalem]{ulem}
\usepackage{mathrsfs}
\usepackage{multirow}
\usepackage{diagbox}
\usepackage{enumitem}
\usepackage{bbm}
 \usepackage[colorlinks=true,citecolor=blue]{hyperref}
\usepackage{amsthm}
\usepackage{color}

\definecolor{DarkGreen}{rgb}{0.2,0.6,0.2}

\definecolor{purple}{rgb}{0.6,0.3,0.8}

\usepackage{natbib}
\usepackage{comment}

\def\d{\mathrm{d}}
\def\laweq{\buildrel \mathrm{d} \over =}

\newcommand{\E}{\mathbb{E}}

\newcommand{\R}{\mathbb{R}}

\newcommand{\p}{\mathbb{P}}

\newcommand{\id}{\mathds{1}}

\renewcommand{\ge}{\geqslant}
\renewcommand{\le}{\leqslant}
\renewcommand{\geq}{\geqslant}
\renewcommand{\leq}{\leqslant}
\renewcommand{\epsilon}{\varepsilon}
\newcommand{\esssup}{\mathrm{ess\mbox{-}sup}}
\newcommand{\essinf}{\mathrm{ess\mbox{-}inf}}
\renewcommand{\cdots}{\dots}
\theoremstyle{plain}
\newtheorem{theorem}{Theorem}
\newtheorem{corollary}{Corollary}
\newtheorem{lemma}{Lemma}
\newtheorem{proposition}{Proposition}
\theoremstyle{definition}
\newtheorem{definition}{Definition}
\newtheorem{example}{Example}

\theoremstyle{remark}
\newtheorem{remark}{Remark}

\usepackage[onehalfspacing]{setspace}

\newcommand{\ES}{\mathrm{ES}}

\usepackage{footmisc}
\begin{document}

\title{On convex order and supermodular order  without finite mean}

\author{Benjamin C\^ot\'e\thanks{Department of Statistics and Actuarial Science, University of Waterloo, Canada.   \texttt{b3cote@uwaterloo.ca}}\and  Ruodu Wang\thanks{Department of Statistics and Actuarial Science, University of Waterloo, Canada.   \texttt{wang@uwaterloo.ca}}}

\maketitle

\begin{abstract}

    Many results on the convex order in the literature were stated for random variables with finite mean. 
    For instance,  a fundamental result in dependence modeling is that the sum of a pair of random random variables  is upper bounded in convex order by that of its comonotonic version  and lower bounded by that of its counter-monotonic version, and all existing proofs of this result require the random variables' expectations to be finite. 
    We show that the above result remains true even when discarding the finite-mean assumption, and obtain several other results on the comparison of infinite-mean random variables via the convex order. To our surprise, we find two deceivingly similar definitions of the convex order,  both of which exist widely in the literature, and they are not equivalent for random variables with infinite mean.  This subtle discrepancy in definitions also applies to the supermodular order, and it gives rise to some incorrect statements, often found in the literature. 

\end{abstract}

\section{Introduction}

The convex order is designed to compare the variability of distributions, and it is one of the most popular notions of stochastic comparison.
The convex order  and its variants, such as increasing convex order or concave order, are prominent tools in many fields; 
for instance, they  can model   attitudes of risk aversion in economics  (\cite{RS70}),  risk aggregation and risk sharing in actuarial science (e.g.,~\cite{DDGKV02} and \cite{DD12}),   
assessment of diversification in quantitative finance  (e.g.,~\cite{FS16} and \cite{MW20}),
and comparison between statistical experiments (\cite{T91}).
Through the classic representation theorem of \cite{S65},  
the convex order is fundamental to the construction of martingales (\cite{HPRY11}) and the theory of
martingale optimal transport (\cite{BHP13}).
Comprehensive textbook treatments of the convex order  can be found in \citet[Chapter~1]{muller2002} and   \citet[Chapter~3]{shaked2007}.

In this paper we explore some hidden aspects of the celebrated notion of convex order, and its related notion of supermodular order. 
Results in the literature on the convex order are often stated for the case of finite-mean random variables. Some authors, including the standard textbook \cite{muller2002}, exclude infinite-mean or undefined-mean random variables in their definitions. Heavy-tailed distributions 
are prominent in many scientific fields. 
In some contexts, the assumption of a finite mean is harmless, but
probabilistic models of infinite mean are encountered in many settings in finance and insurance; see \citet[Section~3]{CW25} for a survey. One naturally wonders whether the assumption of finite expectation only excludes trivial cases or  it makes a real difference to main properties of the convex order. This is the primary purpose of the paper.


We start the discussion by recalling a well-known result on the convex order, which is fundamental in the research area of risk aggregation.
Let $d\ge 2$, $(X_1,\dots,X_d)$ be a random vector with  finite-mean components,
and  $(X_1^{\rm co},\dots,X_n^{\rm co})$ be a comonotonic random vector such that $X_i^{\rm co}\laweq X_i$ for each $i\in \{1,\dots,d\}$.\footnote{Definitions of comonotonicity and counter-monotonicity are formally presented in Section \ref{sect:CXsum-explanations}.} 
We have (e.g., \citet[Theorem 7]{DDGKV02})
\begin{align}
\label{eq:convex-order-intro}
\begin{aligned}
& \E[f(X_1+\dots+X_d)]\le \E[f(X_1^{\rm co}+ \dots+X_d^{\rm co})] 
\\ & 
\mbox{for all convex functions $f$ such that the two expectations are finite.}
\end{aligned}
\end{align}
The relation \eqref{eq:convex-order-intro} is often seen as a special case of a more general relation\footnote{A function   $\varphi:\mathbb{R}^d\to \R$ is \textit{supermodular} if  
$\varphi(\mathbf {x}) + \varphi(\mathbf {y}) \le \varphi(\mathbf {x}\wedge \mathbf {y}) + \varphi(\mathbf {x}\vee \mathbf {y}) \mbox{ for all }\mathbf {x}, \mathbf {y}\in\mathbb{R}^d,$ where $\wedge$ represents taking the componentwise minimum and $\vee$, the componentwise maximum; see Section \ref{sec:discuss}.} 
\begin{align}
\label{eq:convex-order-intro2}
\begin{aligned}
&\E[\varphi(X_1,\dots,X_d)]\le \E[\varphi(X_1^{\rm co}, \dots, X_d^{\rm co})] \\ & 
\mbox{for all supermodular functions $\varphi$ such that the two expectations are finite.}
\end{aligned}
\end{align} 
The relation \eqref{eq:convex-order-intro2} is widely believed to hold true; see e.g.,~\citet[Theorem 3.9.8, Property 5]{muller2002} and we list several other references later.
The connection between \eqref{eq:convex-order-intro} 
and \eqref{eq:convex-order-intro2}
can be seen from the fact that $(x_1,\dots,x_d)\mapsto f(x_1+\dots+x_d)$ is supermodular whenever $f$ is convex. 

Our first, perhaps surprising, example shows that \eqref{eq:convex-order-intro2} is actually not true, at least when $d\ge 3$. As a consequence, \eqref{eq:convex-order-intro} does not directly follow from \eqref{eq:convex-order-intro2}.

\begin{example}
\label{ex:SIMONS}
Consider the random vector $(X,Y,Z)$ where all three random variables follow a standard uniform distribution. Assume that  $(X,Y)$ is counter-monotonic, and $(1/X + 1/Y,Z)$ is counter-monotonic. 
Write $V = 2/Z$ and $W = 1/X + 1/Y$. 
Consider the function $\varphi:\mathbb{R}^{3}\mapsto\mathbb{R}$ given by 
\begin{equation}
\label{eq:COUNTEREXAMPLE}
\varphi(x,y,z) = \frac{1}{x}\id_{\{0<x<1\}} + \frac{1}{y}\id_{\{0<y<1\}} - \frac{2}{z}\id_{\{0<z<1\}}, \quad \text{for }x,y,z\in\mathbb{R}. 
\end{equation}
This function is separable in $3$ variables, and therefore  it is supermodular. We have
\begin{align*}
\mathbb{E}[\varphi(X,Y,Z)] = \mathbb{E}[W-V] = \int_0^1 (F^{-1}_W(u)-F^{-1}_V(u))\d u,
\end{align*}
where $F_W^{-1}$ and $F_{V}^{-1}$ denote  the left quantile functions of $W$  and $V$, respectively,\footnote{Formally, $F_X^{-1}(t) =\inf\{x\in \R: \p(X\le x)\ge t\}$ for $t\in (0,1)$ and any random variable $X$.} and the second equality follows from the comonotonicity of $(W,V)$. One easily finds $F^{-1}_W(u) =  4/({1-u^2})$ and $F^{-1}_V(u) = 2/({1-u})$ for $u\in(0,1)$. Hence,   
\begin{align*}
\mathbb{E}[\varphi(X,Y,Z)] = \int_0^1 \frac{4}{1-u^2} - \frac{2}{1-u}\,\mathrm{d}u = \int_0^1 \frac{2}{1+u}\, \mathrm{d}u = 2\log 2.
\end{align*}
For the comonotonic random vector $(X^{\rm co},Y^{\rm co},Z^{\rm co})$ with standard uniform marginal distributions, $\varphi(X^{\rm co}, Y^{\rm co}, Z^{\rm co})=0$, which yields $\mathbb{E}[\varphi(X^{\rm co}, Y^{\rm co}, Z^{\rm co})] = 0 < \mathbb{E}[\varphi(X, Y, Z)]$. 
Therefore, \eqref{eq:convex-order-intro2} does not hold  for the bounded random vectors $(X,Y,Z)$ and $(X^{\rm co},Y^{\rm co},Z^{\rm co})$.
\end{example}

In Example~\ref{ex:SIMONS}, the unboundedness of $\varphi$ allowed for the uniform random variables at play to be transformed into infinite-mean Pareto random variables, whose sum has an unusual behavior under negative dependence, as recently studied by \cite{chen2024technical}.  Their Theorem~1 states that $W$ stochastically dominates $V$; thus, $\varphi(X,Y,Z)>0$ almost surely, and with no surprise it has a positive expectation.
 A similar example to Example~\ref{ex:SIMONS}, involving Cauchy distributions, can be built from \cite{S77}. Other counter-examples can   be constructed from the fact that the sum of standard  Cauchy random variables can be a non-zero constant; 
 see \citet[Theorem 4.2]{PRWW19}. Although providing useful examples, none of the above papers pointed out that \eqref{eq:convex-order-intro2} fails to hold.
Note that the relation \eqref{eq:convex-order-intro2} also fails for strictly supermodular functions, by adding a very small strictly supermodular term to $\varphi$ in the example.

 Example~\ref{ex:SIMONS} seems to refute many statements of \eqref{eq:convex-order-intro2} in well-cited papers, including \citet[Theorem 2.1(d)]{PW15} (from one of the authors of the current paper), \citet[Theorem 9.A.21]{shaked2007}, \citet[Theorem 3.9.8, Property 5]{muller2002}, \citet[Theorem 3.1]{MS00}, \citet[Proposition 6.3.7]{DDGK05} and \citet[Theorem 6.14]{R13}. Most of these results relied on the upper bound for expectations of supermodular functions provided in \citet[Theorem~5]{tchen1980}. However, requirements on the integrability of the functions used by \cite{tchen1980} were   stricter than \eqref{eq:convex-order-intro2}.
Proper restrictions on $\varphi$ for  \eqref{eq:convex-order-intro2} to hold can be found in optimal transport (OT) theory as in \citet[Theorem 5.10]{Villani} and \citet[Theorem 3.1.2]{R06MT}; these conditions involve some form of regularity and integrability. 
For instance,  the inequality in 
\eqref{eq:convex-order-intro2} holds true for
$\varphi$ that is bounded and lower semicontinuous. 
This issue will be formally  discussed in Section~\ref{sec:discuss}. 

A key feature of Example \ref{ex:SIMONS} is that $\E[\varphi(X,Y,Z)]$ involves terms that have infinite expectations, although these terms may be canceled for some particular dependence structures of $(X,Y,Z)$ as in  Example \ref{ex:SIMONS}.  
If the standard uniform random variables $X,Y,Z$ are independent, then $\E[\varphi(X,Y,Z)]$ is actually undefined, as it involves $\infty-\infty$. 
Therefore, 
the real trouble to cause 
 \eqref{eq:convex-order-intro2}  to fail is the presence of infinite mean. 

 With  \eqref{eq:convex-order-intro2} failing when infinite-mean terms arise, one naturally wonders whether \eqref{eq:convex-order-intro}  also fails when the random variables do not have a finite mean. It turns out that \eqref{eq:convex-order-intro} still holds true, even when we allow infinite expectations of $\E[f(X_1,\dots,X_n)]$ and $\E[f(X_1^{\rm co},\dots,X^{\rm co}_n)]$; this result is presented as Theorem \ref{th:CXm}.
Notably, when comparing two random variables  $X$ and $Y$ in convex order, requiring $\E[f(X)]$ and $\E[f(Y)]$ to be finite or not leads to two   distinct definitions of the convex order, co-existing  in the literature and commonly believed to be equivalent. This issue is formally discussed in  Section~\ref{sect:two-def}.
We show that  both definitions result in the same partial order on the space of finite-mean random variables, but the equivalence does not hold if the means of random variables are allowed to be infinite.

In Section~\ref{sect:conditions}, we give some equivalent conditions for the convex order to hold for infinite-mean random variables, which are slightly different from  common conditions used for finite-mean random variables. 
The  obtained conditions are useful to establish some later results. 

We proceed to study convex order relations between sums of random variables in Section~\ref{sect:CXsum-explanations}. 
In dimension $d=2$, \eqref{eq:convex-order-intro} becomes the upper bound in the following relation
(e.g.,~\citet[Corollary 3.28]{R13}):  For any $X,Y $ with finite mean,
 \begin{align}
 \label{eq:main-intro} 
 \begin{aligned}
&  \E[f(X^{\rm ct}+Y^{\rm ct})] \le \E[f(X+Y) ]\le \E[f(X^{\rm co}+Y^{\rm co})],
\\ & 
\mbox{for all convex functions $f$ such that the three expectations are finite,}
\end{aligned} 
\end{align} 
where $(X^{\rm co},Y^{\rm co})$ is  comonotonic, 
$(X^{\rm ct},Y^{\rm ct})$ is counter-monotonic,
 $X^{\rm co}$, $X^{\rm ct}$,
and $X$ are identically distributed,
and 
 $Y^{\rm co}$, $Y^{\rm ct}$,
and $Y$ are identically distributed. 
 A first version of this result, formulated quite differently, may be traced back at least to \cite{L53}; for a further historical remark, see \cite{PW15}. 
 Proofs of \eqref{eq:main-intro} in the literature 
 rely heavily on the assumption, although sometimes not emphasized, that the means of $X$ and $Y$ are finite.
This is the case, for instance, for \citet[Proposition 2]{KDG00}, \citet[Theorem 7]{DDGKV02} and \citet[Section 3.4]{DDGK05}, all focusing on the bound provided by the comonotonic version.
By removing  the assumption of finiteness of means,
we establish \eqref{eq:convex-order-intro}, \eqref{eq:main-intro}, and a corresponding lower bound in dimension $d\ge 3$, with the greatest generality.

As mentioned earlier, Section \ref{sec:discuss} addresses the supermodular order and its connection to OT theory, where we obtain several results that hold for random variables with infinite mean. As we can see from the discussions in  Section \ref{sec:discuss}, many  questions remain open, including whether \eqref{eq:convex-order-intro2} holds when $d=2$. 
Section \ref{sec:r1} discusses implications of our main results for risk management and decision making.
Whereas short proofs are kept within the corresponding sections, long proofs   are provided in Section~\ref{sect:CXsum-proof-Theorem1}.



\section{Two definitions of convex order}
\label{sect:two-def}
Let $L^0$ be the space of all random variables on an atomless probability space $(\Omega,\mathcal F,\p)$, $L^1$ be the space of all random variables with finite mean, and $L^{\infty}$ be the space of all essentially bounded random variables. We write $X\laweq Y$ to represent that $X$ and $Y$ follow the same distribution. Throughout, by saying that an expectation $\E[X]$ is well-defined, we mean $\E[X_+]<\infty$ or $\E[X_-]<\infty$,
 where $x_+=\max \{x,0\}$ and $x_-=\max\{-x,0\}$ for $x\in \R$. Denote by $\mathcal{U}_{\rm cx}$ the set of all convex functions on $\mathbb{R}$.

Although it may not be obvious to many researchers, 
two definitions of the convex order co-exist in the literature, and they are deceivingly similar.  
 
\begin{definition}
\label{def:cx-A}
For $X,Y\in L^0$,  we say that $Y$ dominates $X$ 
\begin{enumerate}[label=(\roman*), ref=(\roman*)]
    \item \label{item:cx-defA} in \emph{convex order}, denoted by $X\leq_{\rm cx} Y$,
if
$\E[u(X)]\leq \E[u(Y)]$ 
 for all $u\in\mathcal{U}_{\rm cx}$ such that the two expectations are well-defined; 
    \item \label{item:cx-defB} in \emph{$^\dagger\!$convex order}, denoted by $X\leq_{\rm cx}^\dagger Y$, if
$\E[u(X)]\leq \E[u(Y)]$
 for all $u\in\mathcal{U}_{\rm cx}$ such that the two expectations are finite. 
\end{enumerate}
\end{definition}

The definition as per item~\ref{item:cx-defA} is found in \cite{muller2002}, notably; the one as per item~\ref{item:cx-defB} is found in \cite{R13}, \cite{DDGK05} and \cite{shaked2007}. We added a dagger $\dagger$ to dispel possible ambiguity as to which definition we refer to henceforth. 
The difference is whether one must consider functions in $\mathcal{U}_{\rm cx}$ rendering infinite expectations to establish comparisons: one needs to for $\leq_{\rm cx}$ but not for $\leq_{\rm cx}^{\dagger}$. Because in the latter case $\mathbb{E}[u(X)]\le\mathbb{E}[u(Y)]$ is required for a smaller subset of functions $u\in\mathcal{U}_{\rm cx}$, we immediately have $X\leq_{\rm cx}Y\implies X\leq_{\rm cx}^{\dagger}Y$. When comparing random variables on $L^{\infty}$, the definitions are obviously equivalent since all $u\in\mathcal{U}_{\rm cx}$ inevitably produce finite expectations.   

The most plausible cause for this discrepancy to have arisen is that most versions of the convex order's definition state that the expectations must \textit{exist}, which can refer to either expectations be well-defined or finite depending on the authors' wording choices. In \cite{muller2002} and \cite{R06MT}, existence of expectations means they are well-defined;  for \cite{DDGK05} and \cite{shaked2007},
existence means finiteness.\footnote{They
do not explicitly indicate what the existence of expectations signifies, but it can be inferred from other results,  e.g., Property 1.8.9 and its proof in \cite{DDGK05} and Example 1.B.23 in \cite{shaked2007}.} This semantic inconsistency is not the product of either set of textbooks and is in fact not circumscribed to the literature on stochastic orders. Even in the standard probability textbook \citet[Fifth edition]{D19}, while
it is declared that an expectation that exists may be infinite (p.~28),
in some places the existence of an expectation is understood as finiteness (p.~29).

Another plausible cause for the discrepancy is that the nuance may have come off as inconsequential. In the following proposition, we show that the two definitions coincide when means are finite. It is not the case, however, when on the space of all random variables. 


\begin{proposition}
\label{prop:equivalent-def}
    On $L^1$, 
    the two relations
    $\le_{\rm cx}$
    and     $\le^\dagger_{\rm cx}$
    are equivalent.   
On $L^0$,   $\le_{\rm cx}$
 implies  $\le^\dagger_{\rm cx}$,
 but the converse is not true.     
\end{proposition}
\begin{proof}
    As discussed above, the implication $X\le_{\rm cx} Y \implies X\le_{\rm cx}^\dagger Y$ holds under both settings of $L^1$ and $L^0$.
    We first show the converse holds on $L^1$.   By Theorems 1.5.3 and 1.5.7 of \cite{muller2002},
    we have, for $X,Y\in L^1$, 
    \begin{align}
        \label{eq:alter-cx-1}
    X\le_{\rm cx} Y  
    \iff
    \E[X]=\E[Y] \mbox{~and~}
    \E[(X-w)_+] \le \E[(Y-w)_+] \mbox{ for all $w\in \R$.}
    \end{align} 
     By Proposition 3.4.3 of \cite{DDGK05}
    we have, for $X,Y\in L^1$, 
    \begin{align}
        \label{eq:alter-cx-2}
    X\le_{\rm cx}^\dagger Y  
    \iff
    \E[X]=\E[Y] \mbox{~and~}
    \E[(X-w)_+] \le \E[(Y-w)_+] \mbox{ for all $w\in \R$.}
    \end{align} 
   Putting \eqref{eq:alter-cx-1} and \eqref{eq:alter-cx-2} together, we get $    X\le_{\rm cx} Y  
    \iff     X\le_{\rm cx}^\dagger Y$. 
To show that $\le_{\rm cx}^{\dagger}$ does not imply $\le_{\rm cx}$ on $L^0$, 
a counter-example is provided in Example \ref{ex:different-def} below.  
\end{proof}

 Because the $^{\dagger}$convex order gets rid of infinite expectations, whether positive or negative, one can design a comparison scheme where all elements $u\in\mathcal{U}_{\rm cx}$, except for constant functions, are removed because of either random variable. 
 
\begin{example}
\label{ex:different-def}
    Consider two random variables $X\ge 0$  and $Y \le 0 $  with $\mathbb{E}[X]=-\mathbb{E}[Y]=\infty$.      
 Any convex function $u$  with $\mathbb{E}[u(X)]< \infty$ must not have any strictly increasing part,
and that with
$\mathbb{E}[u(Y)]<\infty$ must not have any strictly decreasing part. With both requirements, we are left only with the constant functions,
for which $\mathbb{E}[u(X)] = \mathbb{E}[u(Y)]$ clearly holds, and this implies  both $X\leq_{\rm cx}^{\dagger} Y$  and $Y\leq_{\rm cx}^{\dagger} X$. Meanwhile, for convex functions $u_1:x\mapsto x_+$ and $u_2:x\mapsto x_-$, we note $\mathbb{E}[u_1(X)] = \infty > 0 = \mathbb{E}[u_1(Y)]$ and $\mathbb{E}[u_2(X)] = 0 < \infty = \mathbb{E}[u_2(Y)]$; therefore, neither 
    $X\leq_{\rm cx} Y$ nor $Y\leq_{\rm cx} X$   holds.
\end{example}

   \cite{muller2002} restricted their definition of the convex order to random variables in $L^1$, based on limitations of the convex order beyond $L^1$ identified by \cite{elton1992fusions}.  Given Proposition~\ref{prop:equivalent-def}, \cite{muller2002}, \cite{DDGK05},  \cite{shaked2007}, and \cite{R13}, as well as many others,  all effectively work under the same definition on $L^1$.

Example~\ref{ex:different-def} illustrates that $X,Y\in L^0$ may be comparable with respect to the $^\dagger$convex order even if $\mathbb{E}[X]\neq\mathbb{E}[Y]$, when these means are infinite.
 It moreover showcases that $\leq_{\rm cx}^{\dagger}$ does not satisfy antisymmetry on $L^0$, as $X\le_{\rm cx}^{\dagger}Y$ and $Y\le_{\rm cx}^{\dagger}X$ despite $X\not\stackrel{\rm d}{=}Y$. Neither does the order $\leq_{\rm cx}$ satisfy antisymmetry on $L^0$; see Example~\ref{ex:cauchy} below. 

\begin{example}
\label{ex:cauchy}
    Let $X$ follow a standard Cauchy distribution, and $Y = 2X$.  Recall that the Cauchy distribution has no finite moment and is symmetric; in this vein, no convex function $u$ yields finite expectations $\E[u(X)]$ and $\E[u(Y)]$ besides $u$ constant. If expectations are well-defined, then
    $\mathbb{E}[u(X)] = \mathbb{E}[u(Y)] = \infty$. 
     Hence, $X\leq_{\rm cx}Y$ and $Y\leq_{\rm cx}X$ hold. The two random variables, however, are not identically distributed.  
\end{example}

We should note that for random variables $X$ and $Y$ that do not have a finite mean, $Y \le_{\rm cx} X$ loses the classic interpretation that the distribution of $X$ is more spread out than that of $Y$. For instance, in the above example, $Y=  2 X$ and $X$ has a symmetric distribution with respect to $0$. Thus, the distribution of $Y$ is more spread out than that of $X$, despite $Y\le_{\rm cx}X$.
 
The following example was inspired from the discussion in \cite{M97integral} and showcases that $\leq_{\rm cx}^{\dagger}$ does not satisfy transitivity.  
\begin{example}
    \label{ex:transitivity}
    Let $X\in L^0$ follow a standard Cauchy distribution, and $Y,Z\in L^1$ be such that $\mathbb{E}[Y]\neq\mathbb{E}[Z]$. Because $\E[u(X)]$ is finite for no $u\in\mathcal{U}_{\rm cx}$ other than constant ones, we have the chain of relations $ Y\leq_{\rm cx}^{\dagger} X \leq_{\rm cx}^{\dagger} Z\leq_{\rm cx}^{\dagger} X \leq_{\rm cx}^{\dagger} Y$. Yet, from \eqref{eq:alter-cx-2}, $Y\not \leq_{\rm cx}^{\dagger} Z$  and $Z\not \leq_{\rm cx}^{\dagger} Y$.
\end{example}
The flaw exhibited in Example~\ref{ex:transitivity} does not spread to $\leq_{\rm cx}$, as we attest by the next proposition. 
 Still, $\leq_{\rm cx}$ is not a partial order on $L^0$ because it fails to satisfy antisymmetry; in contrast, $\le_{\rm cx}$ and $\le_{\rm cx}^{\dagger}$ are partial orders on $L^1$.  
\begin{proposition}
    \label{prop:transitivity}
    The relation $\leq_{\rm cx}$ is a preorder on $L^0$.
\end{proposition}
\begin{proof}
     The reflexivity of $\leq_{\rm cx}$ follows trivially from  definition; remains to show transitivity. Consider $X,Y,Z \in L^0$ and suppose $X\leq_{\rm cx} Y$ and $Y\leq_{\rm cx} Z$.
     Let $u\in\mathcal{U}_{\rm cx}$ be such that $\E[u(X)]$ and $\E[u(Z)]$ are well-defined. 
     If $\E[u(Y)]$ is also well-defined, then $\mathbb{E}[u(X)]\leq \mathbb{E}[u(Y)] \leq \mathbb{E}[u(Z)]$.
     If $\E[u(Y)]$ is not well-defined, then
  $\infty=\mathbb{E}[u(Y)_+] \leq \mathbb{E}[u(Z)_+]$ since the function $x\mapsto u(x)_+$ is   convex.
Hence, $\E[u(Z)]=\infty$ and thus  $\mathbb{E}[u(X)] \leq \mathbb{E}[u(Z)]$. 
This shows $X\le_{\rm cx} Z$.
     \end{proof}

The convex order falls within the more general class of integral stochastic orders (\cite{W86,M97integral}). An integral stochastic order 
$\leq_{\mathfrak{F}}$, indexed by a set $ \mathfrak{F}$ of real-valued functions,
is 
defined on a set  $\mathcal{L}$  of  random vectors taking values in $\R^d$ via 
\begin{equation}
\label{eq:integralstoorder}
 X\le_{\rm \mathfrak{F}} Y \iff \mathbb{E}[u(X)]\leq \mathbb{E}[u(Y)]\text{ for all }u\in\mathfrak{F};\quad  \quad\quad X,Y\in\mathcal{L}.
\end{equation} 
 The set $\mathcal L$ is called the domain of $\leq_{\mathfrak{F}}$.
  The convex order and the supermodular order (see Section~\ref{sec:discuss}) are     instances of integral stochastic orders with $d=1$. 
%
To avoid situations such as in Example~\ref{ex:transitivity} and ensure integral orders be preorders, the setting of \cite{M97integral} additionally requires 
\begin{equation}
\label{eq:bfrak}
\mbox{there exists $b_\mathfrak F:\R^d\to [1,\infty)$ such that }
    \sup_{x\in \R^d}\frac{|u(x)|}{b_{\mathfrak{F}}(x)} < \infty   \text{ for all } u\in \mathfrak{F},
\end{equation}
and excludes from the   set $\mathcal{L}$ every   $X$ with $\mathbb{E}[b_{\mathfrak{F}}(X)]=\infty$. Expectations in \eqref{eq:integralstoorder} are then finite for all $X,Y\in\mathcal{L}$. 
The set $\mathcal{L}$ is restricted by $\mathfrak{F}$  via the function $b_{\mathfrak{F}}$.   
Note that our definitions of $\leq_{\rm cx}$ and $\leq_{\rm cx}^{\dagger}$ do not restrict the set $\mathcal L$. When defined on $L^1$,  $\leq_{\rm cx}$ and $\leq_{\rm cx}^{\dagger}$ are  consistent with the above restriction: Taking $ {\mathfrak{F}}=\{I_\R,-I_\R\}\cup\{x\mapsto(x-w)_+:w\in\mathbb{R}\}$, where $I_\R$ is the identity on $\R$, one can take $b_{\mathfrak{F}} = I_\R\vee 1$ and $\mathcal{L}=L^1$ in \eqref{eq:bfrak}. By \eqref{eq:alter-cx-1} and \eqref{eq:alter-cx-2}, the relation $\leq_{\mathfrak{F}}$  is equivalent to $\leq_{\rm cx}$ and $\leq_{\rm cx}^{\dagger}$ on $L^1$.

An advantage of the restriction  \eqref{eq:bfrak} for integral stochastic orders is that one never has to worry about infinite expectations and the troubles arising along. The drawbacks are limitations on the applicability of these orders.
For instance, one cannot consider the complete set of random variables in $L^0$ unless $\mathfrak{F}$ only contains bounded functions; note that bounded convex functions are constant.
As we will discuss in the upcoming sections, $\leq_{\rm cx}$ continues to satisfy main properties of the convex order on $L^0$. Therefore, narrowing arbitrarily the domain for $\leq_{\rm cx}$ seems unnecessary. This is indeed one of the main arguments of the paper: one need not limit to $L^1$. In Section~\ref{sec:discuss}, we will discuss how, on the contrary, a narrowing of the domain can be beneficial for the supermodular order.

\section{Equivalent conditions for the convex order}
\label{sect:conditions}

There are many conditions to verify convex order. A common equivalent condition for $X\le_{\rm cx } Y$ when $X,Y\in L^1$ is  given in 
\eqref{eq:alter-cx-1}; another one is
(\citet[Theorem 3.A.5]{shaked2007})
\begin{align}
    \label{eq:alter-cx}
\E[X]=\E[Y]
    \mbox{~~~ and ~~~}\int_p^1 F_X^{-1}(t)\d t \le \int_p^1 F_Y^{-1}(t)\d t \mbox{~for all $p\in (0,1)$}.
\end{align}  
Condition \eqref{eq:alter-cx}, as well as the second condition in \eqref{eq:alter-cx-1}, is not sufficient for either $X\le_{\rm cx } Y$  or $X\le_{\rm cx }^\dagger Y$  with $X,Y\in L^0$, because \eqref{eq:alter-cx}  holds as soon as $\E[X_+]=\E[Y_+]=\infty$,  $\E[X_-]<\infty$,
and $\E[Y_-]<\infty$, but   $X\le_{\rm cx }^\dagger Y$ and $X\le_{\rm cx } Y$  fail to hold  when $\E[X_-]>\E[Y_-]$ because $x\mapsto x_-$ is convex. 
Moreover, even when the means are well-defined, \eqref{eq:alter-cx} is not necessary for $X\le_{\rm cx }^\dagger Y$ as we see in Example \ref{ex:different-def}.  
Nevertheless,  similar conditions to \eqref{eq:alter-cx-1}  and \eqref{eq:alter-cx}   turn  out to be equivalent to $X\le_{\rm cx } Y$  for all $X,Y\in L^0$.
\begin{theorem}
\label{th:ES-based-condition}
For $X,Y\in L^0,$  the following are equivalent: 
\begin{enumerate}[label= \rm(\roman*), ref=(\roman*)]
\item \label{item:cond-1}
$X\le_{\rm cx } Y$; 
\item \label{item:cond-2} it holds that
\begin{align}
    \label{eq:cond-ES}
\int_0^p F_X^{-1}(t)\d t \ge \int_0^p F_Y^{-1}(t)\d t
    \mbox{~~and~~}\int_p^1 F_X^{-1}(t)\d t \le \int_p^1 F_Y^{-1}(t)\d t \mbox{~~~for all $p\in (0,1)$};
\end{align}  

\item \label{item:cond-3} it holds that
\begin{equation}
    \mathbb{E}[(X-w)_-]\leq \mathbb{E}[(Y-w)_-]  \mbox{~~and~~}  \mathbb{E}[(X-w)_+]\leq \mathbb{E}[(Y-w)_+]   \mbox{~~~for all }w\in\mathbb{R}.  \label{eq:stoploss}
\end{equation} 
\end{enumerate}
\end{theorem}
We note that, regardless of whether $X$ and $Y$ have finite mean, the integrals and expectations in \eqref{eq:cond-ES} and \eqref{eq:stoploss} are always well-defined, as each of them integrates a function that is bounded from one side. 
The proof of 
Theorem~\ref{th:ES-based-condition} relies on Lemma~\ref{th:lemma} below, which gives a simplification of the convex order when one of the two random variables at comparison is outside $L^1$. Let us present a few more stochastic orders first.  

\begin{definition} Let $\mathcal{U}_{\rm st}$ (resp.~$\mathcal{U}_{\rm icx}$, $\mathcal{U}_{\rm dcx}$) be the set of all increasing (resp.~increasing convex, decreasing convex) real functions on $\R$. For random variables $X, Y\in L^0$ and $*\in\{\mathrm{st}, \mathrm{icx}, \mathrm{dcx}\}$, we write $X\leq_* Y$ if $\E[u(X)]\leq\E[u(Y)]$ for all $u\in\mathcal{U}_{*}$ such that the expectations are well-defined. 
\end{definition}

\begin{lemma}
	\label{th:lemma}
	For $X, Y\in L^0$, 
	\begin{enumerate}[label = \rm(\roman*), ref=(\roman*)]
		 \item \label{item:positive} if $\E[Y_+] = \infty$, then $X\leq_{\rm cx}Y$ is equivalent to $X\leq_{\rm dcx}Y$;  
		 \item \label{item:negative} if  $\E[Y_-] = \infty$, then $X\leq_{\rm cx}Y$ is equivalent to $X\leq_{\rm icx}Y$; 
		 \item \label{item:trivial} if $\E[Y_+] =\E[Y_-] = \infty$, then $Y\leq_{\rm cx}^\dagger X\leq_{\rm cx}Y$  always holds.  
	\end{enumerate}
\end{lemma}

By Lemma~\ref{th:lemma}, for random variables $X,Y\in L^0\setminus L^1$,  one may observe $X \geq_{\rm st} Y$ meanwhile $X \leq_{\rm cx} Y$. This situation cannot occur for $X,Y$ with finite means (unless $X\stackrel{\rm d}{=}Y$) since convex ordering then requires $\E[X] = \E[Y]$, from \eqref{eq:alter-cx}. 
 The convex order has been mainly designed and hitherto employed to compare the variability of distributions. On $L^1$, the variability may be vaguely interpreted as some form of deviation from the mean. 
 When the mean is infinite, for instance when $\E[Y]=\infty$, 
 this deviation is larger when the distribution is stochastically smaller, and hence $\le_{\rm cx}$ becomes $\le_{\rm dcx}$. Therefore, the interpretation of convex order as a comparative tool for deviation from the mean remains valid even on $L^0$, as we see in parts (i) and (ii) of  Lemma \ref{th:lemma}.
 

Each of
\eqref{eq:cond-ES} and \eqref{eq:stoploss} is sufficient, but not necessary, for $X\le_{\rm cx }^\dagger Y$ on $L^0$.   Example \ref{ex:different-def} illustrates one instance in which \eqref{eq:cond-ES}  and \eqref{eq:stoploss} fail  but $X\le_{\rm cx }^\dagger Y$ holds. The following theorem provides necessary conditions to establish $^{\dagger}$convex order in the setting of $L^0$. 
\begin{theorem}
\label{th:dagger-ES-based-condition}
For $X,Y\in L^0$, the following are equivalent:
\begin{enumerate}[label= \rm(\roman*), ref=(\roman*)]
\item \label{item:dagger-cond-1} $X\le_{\rm cx}^{\dagger} Y$;
\item \label{item:dagger-cond-2} 
 the inequalities in \eqref{eq:cond-ES} hold with the relaxation that $\infty\le x $ and $x \le-\infty$ are treated as true for all $x\in [-\infty,\infty]$;
\item \label{item:dagger-cond-3} the inequalities in \eqref{eq:stoploss} hold with the relaxation that $\infty\le x $ and $x \le-\infty$ are treated as true for all $x\in [-\infty,\infty]$.
\end{enumerate}
\end{theorem}
Comparing Theorems \ref{th:ES-based-condition} and \ref{th:dagger-ES-based-condition}, 
we can get the conclusion in
Proposition \ref{prop:equivalent-def} that $\le_{\rm cx}$ and $\le_{\rm cx}^\dagger$ are equivalent on $L^1$ but not on $L^0$.

The conditions' relaxation in items \ref{item:dagger-cond-2} and \ref{item:dagger-cond-3} may be deemed unnatural, and it gives rise to peculiar cases where a positive random variable may be comparable to a negative random variable in $^{\dagger}$convex order, as in Example~\ref{ex:different-def}. 
To exclude these incongruent cases, one might prefer using $\leq_{\rm cx}$ over $\leq_{\rm cx}^{\dagger}$ in the setting of $L^0$. In other words, when the means of $X$ and $Y$ are infinite, the $^\dagger$convex order is not very well-suited to establish comparisons as its requirement on the finiteness of $\E[u(X)]$ and $\E[u(Y)]$ excludes too many $u\in\mathcal{U}_{\rm cx}$ for the comparison $X\leq_{\rm cx}^{\dagger} Y$ to be practically useful.

 Theorems \ref{th:ES-based-condition} and \ref{th:dagger-ES-based-condition}
  also   imply that both $\le_{\rm cx} $ and
  $\le_{\rm cx}^\dagger $
 are closed under distributional mixtures on $L^0$. Note that this does not directly follow from the definition (whereas the case on $L^\infty$ follows directly by definition), because the functions $u\in \mathcal U_{\rm cx}$ that needs to be checked change when distributions are mixed.

\section{Convex order between sums of random variables}
\label{sect:CXsum-explanations}

We first define a few classic concepts of dependence that are relevant for convex order relations. 
 A pair of random variables $(X,Y)$ is \emph{comonotonic} 
 if $X=f(Z)$ and $Y=g(Z)$ almost surely for some random variable $Z$ and increasing functions $f$ and $g$. 
     A pair of random variables $(X,Y)$ is \emph{counter-monotonic} 
 if $(X,-Y)$ is comonotonic.  
	Comonotonicity and counter-monotonicity for random vectors  are meant pairwise. For a random vector $(X_1,\dots,X_d)$, we say that $(X_1',\dots,X_d')$ is its \emph{comonotonic} (resp.~\emph{counter-monotonic}) \emph{version}
if $(X_1',\dots,X_d')$ is comonotonic (resp.~counter-monotonic) and $X_i\laweq X_i'$ for each $i\in [d]=\{1,\ldots,d\}$.
For any given random vector, 
its comonotonic version 
always exists, whereas its counter-monotonic version may not exist for $d\ge 3$ (\cite{D72}).

As mentioned in the Introduction, comonotonicity and 
  counter-monotonicity are connected to convex order  on $L^1$ through $ X^{\rm ct}+Y^{\rm ct} \le _{\rm cx} X+Y \le _{\rm cx} X^{\rm co}+Y^{\rm co}$.
  Our next result shows that this connection continues to hold on $L^0$. 

\begin{theorem}
\label{th:CX}
For any  $X,Y\in L^0$, we have 
 \begin{align}
 \label{eq:main}
 X^{\rm ct}+Y^{\rm ct} \le _{\rm cx} X+Y \le _{\rm cx} X^{\rm co}+Y^{\rm co},
\end{align} 
where $(X^{\rm co},Y^{\rm co})$ is a comonotonic version of $(X,Y)$
and 
$(X^{\rm ct},Y^{\rm ct})$ is a counter-monotonic version of $(X,Y)$. 
\end{theorem}

Since $Z\le_{\rm cx}W\implies Z\le_{\rm cx}^{\dagger}W$, 
the result of Theorem~\ref{th:CX} is also valid for $\le_{\rm cx}^{\dagger}$. 
 Another implication of Theorem~\ref{th:CX} is the following corollary, previously shown by \cite{S77}.

 \begin{corollary}[\cite{S77}] \label{coro:Simons} Consider $X,Y,Z, W\in L^0$ and suppose $X\laweq Z$ and $Y\laweq W$. If $\E[X+Y]$ and $\E[Z+W]$ are both well-defined, then $\E[X+Y]=\E[Z+W]$.
 \end{corollary}
 \begin{proof}
Note that the convex order relation $X^{\rm ct}+Y^{\rm ct}\le_{\rm cx} X+Y$ from Theorem \ref{th:CX} implies $\E[(X^{\rm ct}+Y^{\rm ct})_+] \le \E[(X +Y )_+]$
and $\E[(X^{\rm ct}+Y^{\rm ct})_-] \le \E[(X +Y )_-]$ by \eqref{eq:stoploss}.
Hence, $\E[X^{\rm ct}+Y^{\rm ct}]$ is well-defined  because $\E[X+Y]$  is well-defined.
With this  and the relations $X^{\rm ct}+Y^{\rm ct}\le_{\rm cx} X+Y$ and  $X^{\rm ct}+Y^{\rm ct}\le_{\rm cx} Z+W
$, we get $\E[X^{\rm ct}+Y^{\rm ct}]=\E[X+Y]=\E[Z+W]$ since both $x\mapsto x$ and $x\mapsto -x$ are convex. 
 \end{proof}
 Note that the proof of Corollary~\ref{coro:Simons} requires the ordering in \eqref{eq:main} for $\leq_{\rm cx}$ specifically (rather than for $\leq_{\rm cx}^{\dagger}$), as we compare possibly infinite expectations.
 While Corollary~\ref{coro:Simons} appears evident on $L^1$, it is not trivial on $L^0$ since the relation $\mathbb{E}[X+Y]=\mathbb{E}[X]+\mathbb{E}[Y]$ may not hold because one random variable's mean is $\infty$ and the other's is $-\infty$ or because at least one has an undefined mean. 
 The noteworthiness of the result hence lies in that two random variables' infinite means cannot offset each other in many different ways so that their sum's mean varies upon the dependence scheme. This is curiously not the case when summing three or more random variables, as exhibited in Example~\ref{ex:SIMONS} and discussed by \cite{S77}. 
The proof techniques of \cite{S77} are similar to the ones of Theorem \ref{th:CX}, and in fact they inspired our proof. 

We next compare sums of comonotonic random variables. This relation on $L^1,$ stated for $d \ge 2$ random variables, was presented in  \citet[Corollary 1]{DDGKV02}.

\begin{theorem}
 \label{th:cx-cx}
     For   $X,Y, Z,W\in L^0$,
   if 
 $X\leq_{\rm cx} Z$ and $Y\leq_{\rm cx} W$, then  
     $$
X+Y\le_{\rm cx }Z^{\rm co} + W^{\rm co},
$$
where $(Z^{\rm co},W^{\rm co})$ is a comonotonic version of $(Z,W)$. 
 \end{theorem}
\begin{proof}
Using Theorem \ref{th:CX}, it suffices to prove the result when $(X,Y)$ is comonotonic. 
    This result follows directly by using the equivalent condition \eqref{eq:cond-ES} in Theorem~\ref{th:ES-based-condition} and the relation $F_{X+Y}^{-1}(t) = F_{X}^{-1}(t)+F_{Y}^{-1}(t)$, $t\in(0,1)$, for comonotonic random variables. 
\end{proof}
\begin{remark}
    The same conclusion in Theorem \ref{th:cx-cx} holds also for $\le_{\rm cx}^\dagger$ following a similar proof, which we omit. 
\end{remark}


A potential counterpart to Theorem~\ref{th:cx-cx}, namely $X^{\rm ct}+Y^{\rm ct} \le_{\rm cx} Z+W$ under the same condition,  fails to hold even for random variables in $L^\infty$; see
 \citet[Example 1]{CDLT14}.
An intuitive explanation is that $Z+W$ may be a constant, but $X^{\rm ct}+Y^{\rm ct} $ is not necessarily a constant under the given conditions. 
 
Using Theorem~\ref{th:cx-cx}, we can immediately show that the upper bound in \eqref{eq:main} can be generalized to an arbitrary dimension $d$ on $L^0$. 
   \begin{theorem}
\label{th:CXm} Let $d\in\mathbb{N}$. 
For any  $X_1,\dots,X_d\in L^0$, we have 
 \begin{align}
 \label{eq:main-m}
	X_1 + \cdots +X_d \leq_{\rm cx} X_1^{\rm co} + \cdots + X_d^{\rm co}, 
\end{align} 
where $(X_1^{\rm co},\dots,X_d^{\rm co})$ is a comonotonic version of $(X_1,\dots,X_d)$.
\end{theorem}
\begin{proof}Because sums of any number a comonotonic random vector's components are comonotonic, the result is obtained by combining iterative applications of Theorems~\ref{th:CX} and \ref{th:cx-cx}. 
\end{proof}



As mentioned in Section \ref{sect:conditions}, for $X,Y\in L^0\setminus L^1$, it is possible that  $X\leq_{\rm cx}Y$ meanwhile $Y\leq_{\rm st} X$.  
The implication of this phenomenon in risk management and decision making
is discussed in Section \ref{sec:r1}.

In the next result, we show that  the lower bound in \eqref{eq:main} can also be generalized to an arbitrary dimension $d$.
This result requires a separate proof because there is no counter-monotonic counterpart of Theorem \ref{th:cx-cx} (which is used to prove Theorem \ref{th:CXm}). 
When $d\ge 3$ and $X_1,\dots,X_d$ have at least $3$ non-degenerate components, this result follows from a more general conclusion in Theorem \ref{th:counter-SM}, addressed in the next section, and in the setting of $L^1$ it is given in \citet[Theorem 5.1]{CL14}.
\begin{theorem} 
\label{th:CXm-counter-monotonic}
Let $d\in\mathbb{N}$. For $X_1,\ldots,X_d\in L^0$, we have
\begin{equation*}
X_1^{\rm ct} + \cdots + X_d^{\rm ct} \le_{\rm cx} X_1 + \cdots + X_d,
\end{equation*}
if a counter-monotonic version $(X_{1}^{\rm ct}, \ldots,X_d^{\rm ct})$ of $(X_1,\ldots,X_d)$ exists.
\end{theorem}

Let us repeat that $\leq_{\rm cx}$ implies $\leq_{\rm cx}^{\dagger}$, and hence, the statements of Theorems~\ref{th:CXm} and~\ref{th:CXm-counter-monotonic} also hold true for the $^\dagger$convex order. 
If the marginal distributions of $(X_1,\ldots,X_d)$ do not allow for a counter-monotonic version to exist, the set of
possible distributions of $X_1+\cdots+X_d$ with given marginal distributions of $X_1,\dots,X_d$ may possibly not admit a smallest element with respect to $\leq_{\rm cx}$ or $\leq_{\rm cx}^{\dagger}$, as discussed in Section 3.2 of \cite{bernard2014risk} with a discrete example.

\section{Optimal transport and supermodular order}
\label{sec:discuss}

Theorem~\ref{th:CX} can be interpreted as a result in OT with infinite mean. For given   distributions of $X,Y\in L^0$,  the classical Kantorovich OT problem is to find a joint distribution of $(X,Y)$, called a transport plan, that minimizes the expectation $\mathbb{E}[c(X,Y)]$ for a given real-valued cost function $c$ on $\mathbb{R}^2$.  
It is straightforward to check that functions of the form $(x,y)\mapsto u(x+y)$ are supermodular for $u\in\mathcal{U}_{\rm cx}$; see \citet[Theorem 3.1]{M97}. 
Hence,   for any cost function of the form $c:(x,y)\mapsto u(x+y)$ with $u\in\mathcal{U}_{\rm cx}$, the optimal transport plan is given by the distribution of $(X^{\rm ct},Y^{\rm ct})$. In the same vein, the distribution of $(X^{\rm co}, Y^{\rm co})$ is the optimal transport plan for functions of the form $c:(x,y)\mapsto u(x-y)$ with $u\in\mathcal{U}_{\rm cx}$. The quadratic optimal cost $(x-y)^2$ and the absolute value cost $|x-y|$ are examples of such functions. Theorem~\ref{th:CX} implies that the above conclusions hold even if $X$ and $Y$ do not have a mass barycenter. 
The statement of Theorem~\ref{th:CX} connects to a fundamental OT result for supermodular cost functions, as we discuss below. 
 
Recall that a function   $\varphi:\mathbb{R}^d\to \R$ is \textit{supermodular} if  
$$\varphi(\mathbf {x}) + \varphi(\mathbf {y}) \le \varphi(\mathbf {x}\wedge \mathbf {y}) + \varphi(\mathbf {x}\vee \mathbf {y}) \mbox{ for all }\mathbf {x}, \mathbf {y}\in\mathbb{R}^d,$$ where $\wedge$ represents taking the componentwise minimum and $\vee$, the componentwise maximum. We denote by $\Phi_{\rm sm}^d$ the set of all supermodular functions on $\R^d$.  

\begin{definition}
\label{def:sm-A}
For $\mathbf{X} , \mathbf{Y}  \in (L^0)^d$, $d\in\mathbb{N}$,  we say that $\mathbf {Y}$ dominates $\mathbf {X}$ 
\begin{enumerate}[label=(\roman*), ref=(\roman*)]
    \item \label{item:sm-defA} in \emph{supermodular order}, denoted by $\mathbf {X} \leq_{\rm sm} \mathbf {Y}$,
if
$\E[\varphi(\mathbf {X})]\leq \E[\varphi(\mathbf{Y})]$ 
 for all $\varphi\in {{\Phi}}^{d}_{\rm sm}$ such that the two expectations are well-defined; 
     \item \label{item:sm-defB} in \emph{$^\dagger\!$supermodular order}, denoted by $\mathbf {X} \leq_{\rm sm}^{\dagger} \mathbf {Y}$,
if
$\E[\varphi(\mathbf {X})]\leq \E[\varphi(\mathbf {Y})]$ 
 for all $\varphi\in {{\Phi}}^{d}_{\rm sm}$ such that the two expectations are finite; 
 \item \label{item:sm-defC} in \emph{$^\diamond\!$supermodular order}, denoted by $\mathbf {X} \leq_{\rm sm}^{\diamond} \mathbf {Y}$,
if
$\E[\varphi(\mathbf {X})]\leq \E[\varphi(\mathbf {Y})]$ 
 for all bounded and right-continuous $\varphi\in {{\Phi}}^{d}_{\rm sm}$;
 
 \item \label{item:conc} in \emph{concordance order}, denoted by $\mathbf {X}\le_{\rm c}\mathbf {Y}$, if $F_{\mathbf {X}}(\mathbf {x}) \le F_{\mathbf {Y}}(\mathbf {x})$ and $\overline{F}_{\mathbf {X}}(\mathbf {x}) \le \overline{F}_{\mathbf Y}(\mathbf {x})$ for all $\mathbf {x}\in\mathbb{R}^d$, where $F_{\mathbf{X}}$ and $F_{\mathbf{Y}}$ are the joint distribution functions of $\mathbf{X}$ and $\mathbf{Y}$, respectively, and $\overline{F}_{\mathbf{X}}$ and $\overline{F}_{\mathbf{Y}}$ are their joint survival functions.  
\end{enumerate}
\end{definition}

Note that there was no value in defining a ``$^{\diamond}$convex'' order in the previous sections because bounded convex functions on $\R$ are constant. In the $^\diamond$supermodular order, we required the right continuity of $\varphi$, so that it is consistent with \cite{tchen1980} and \cite{R06MT}. 


 \begin{remark}
The restriction \eqref{eq:bfrak} for integral stochastic orders in Section~\ref{sect:two-def} allows for a continuum of orders ranging from $\leq_{\rm sm}^{\diamond}$ to $\leq_{\rm sm}$, depending on the function $b_{\mathfrak{F}}$ circumscribing the set  ${\mathfrak{F}}$ considered in \eqref{eq:integralstoorder}. The domain embedding the orders varies accordingly. While the domain of $\leq_{\rm sm}^{\diamond}$ is indeed $(L^0)^d$, for $\leq_{\rm sm}$ the domain would reduce to random vectors following a discrete distribution with a finite numbers of atoms, which is inconvenient to work with. Since most authors usually define supermodular order or $^{\dagger}$supermodular order on larger spaces, such as $(L^{\infty})^d$, $(L^{1})^d$ or $(L^0)^d$, we will continue to do so in this section and discuss the problems that consequently arise. 
\end{remark}

For any choice of $w_1,\ldots,w_n\in \R$, functions $\mathbf {x}\mapsto\prod_{i=1}^d \id_{\{x_i\leq w_i\}}$ and $\mathbf {x}\mapsto\prod_{i=1}^d \id_{\{x_i > w_i\}}$ are supermodular, bounded and right-continuous. The expected values of these functions respectively yield the joint distribution function and the joint survival function and are evidently always finite. We therefore have the following implications 
\begin{equation}
    \mathbf {X}\le_{\rm sm}\mathbf {Y}\implies \mathbf {X}\le_{\rm sm}^{\dagger}\mathbf {Y}\implies \mathbf {X}\le_{\rm sm}^{\diamond}\mathbf {Y} \implies \mathbf {X}\le_{\rm c}\mathbf {Y} \quad \text{for any }\mathbf {X},\mathbf {Y}\in(L^0)^d, 
    \label{eq:implications-sm}
\end{equation}
since, as for the convex order, $\le_{\rm sm}^{\dagger}$ considers a smaller subset of functions than $\le_{\rm sm}$ within $\Phi^{d}_{\rm sm}$; idem for $\leq_{\rm sm}^{\diamond}$. 
Note that none of the reverse implications in \eqref{eq:implications-sm} are obvious to be true. Indeed, as we will see below, the reverse directions of the  second and the third implications do not hold in general. We are not aware of whether the reverse direction of the first implication in \eqref{eq:implications-sm} is true.

From the distribution functions and survival functions orderings in the definition of concordance, one immediately sees that $\mathbf {X}\leq_{c}\mathbf {Y}$ 
implies $F_{X_i}(x) = F_{Y_i}(x)$, for all $x\in\mathbb{R}$ and $i\in[d]$. 
Hence, by \eqref{eq:implications-sm}, all the orders in Definition~\ref{def:sm-A} require $\mathbf{X}$ and $\mathbf{Y}$ to have identical marginal distributions.

In the same vein as our discussions in Section~\ref{sect:CXsum-explanations} for the convex and $^\dagger$convex orders, we examine whether, for any random vectors, a comonotonic version and a counter-monotonic version provide an upper and a lower bound with respect to $\leq_{\rm sm}$, $\leq_{\rm sm}^{\dagger}$, $\leq_{\rm sm}^{\diamond}$ and $\leq_{\rm c}$. First, it is straightforward to check for $\leq_{\rm c}$, which follows directly from the Fr\'echet-Hoeffding bounds.


\begin{proposition}
\label{prop:conc}
   Let $d\geq 2$. For any random vector  $\mathbf{X} \in (L^0)^d$, we have $\mathbf{X} \leq_{\rm c} \mathbf{X}^{\rm co}$, where 
$\mathbf{X}^{\rm co}$ is a comonotonic version of $\mathbf{X}$.
Moreover, $\mathbf{X}^{\rm ct} \leq_{\rm c} \mathbf{X}$, where 
$\mathbf{X}^{\rm ct}$ is a counter-monotonic version of $\mathbf X$, if it exists. 
\end{proposition}
The bounds in Proposition~\ref{prop:conc} transpose directly to $\le_{\rm sm}^{\diamond}$ in dimension 2 thanks to the following equivalence. 
\begin{proposition}
\label{prop:equivalent-conc-diamond}
    On $(L^0)^2$, the two relations $\le_{\rm c}$ and $\le_{\rm sm}^{\diamond}$ are equivalent. They are not equivalent when $d\geq 3$ on  $(L^\infty)^d$. 
\end{proposition}
\begin{proof}
    In dimension $d=2$, by \citet[Corollary 2.1]{tchen1980}, for any $X,Y,Z,W\in L^0$, the ordering $F_{X,Y}\leq F_{Z,W}$ implies $\mathbb{E}[\varphi(X,Y)]\leq\mathbb{E}[\varphi(Z,W)]$ for all bounded, right-continuous $\varphi\in\Phi_{\rm sm}^2$. 
For an example where $\le_{\rm c} \,\,\,\,\not\!\!\!\!\implies \le_{\rm sm}^{\diamond}$ in dimension $d=3$, see \citet[Example~1]{M97}, which also generalizes to $d\ge3$.
\end{proof}

Proposition~\ref{prop:equivalent-conc-diamond} is well-known in the literature, although references, while relying on \cite{tchen1980}'s proof, often omit his requirements on $\varphi$; notably, $\varphi(X^*,Y^*)$ must be uniformly integrable for all $(X^*,Y^*)$ in the Fr\'echet class of $(X,Y)$, that is, the class of random vectors with the specified marginal distributions. A similar result can be found in \citet[Theorem 3.1.2 (a)]{R06MT} with the lighter requirement of $\varphi(X,Y)$, $\varphi(Z, W)$, and also $\varphi(X^{\perp},Y^{\perp})$ to produce finite expectations, where $X^{\perp}\laweq X$, $Y^{\perp}\laweq Y$ and $X^{\perp}$ and $Y^{\perp}$ are independent. Hence, one cannot quite establish the equivalence between $\leq_{\rm c}$ and $\le_{\rm sm}^{\dagger}$ from their result. 

Combining Propositions \ref{prop:conc} and \ref{prop:equivalent-conc-diamond}, we immediately arrive at the following conclusion. 
 \begin{proposition}
 \label{th:prop-sm}
  For any $X,Y\in L^0$, we have
  \begin{equation}
  (X^{\rm ct},Y^{\rm ct}) \le_{\rm sm}^{\diamond} (X,Y) \le_{\rm sm}^{\diamond}  (X^{\rm co},Y^{\rm co}) \label{eq:sm-main}
  \end{equation}
  where $(X^{\rm co}, Y^{\rm co})$ is a comonotonic version of $(X,Y)$ and $(X^{\rm ct},Y^{\rm ct})$ is a counter-monotonic version of $(X,Y)$.  
 \end{proposition}

From an OT perspective, Proposition~\ref{th:prop-sm} amounts to the statement that a counter-monotonic transport map is a minimizer for any right-continuous and bounded supermodular cost function $\varphi\in\Phi_{\rm sm}^2$. 
Indeed, this can be extended to lower semicontinuous
supermodular functions $\varphi$ that are bounded from below by a separable form of integrable functions; see \citet[Theorem 5.10]{Villani}.
As this requirement is not satisfied by all functions generating  $\leq_{\rm sm}$ and $\leq_{\rm sm}^\dagger$, 
it is unclear to us whether \eqref{eq:sm-main}   also holds for $\leq_{\rm sm}$ or $\leq_{\rm sm}^\dagger$.

More generally, we do not know if $\le_{\rm sm}$, $\le^{\dagger}_{\rm sm}$ and $\le^{\diamond}_{\rm sm}$ all render the same partial ordering  on any one of $(L^{\infty})^2$, $(L^1)^2$ or $(L^0)^2$. We suspect this could be the case. Corollary~\ref{coro:Simons} gives evidence towards this conjecture because one thus cannot obtain a bivariate counterpart to Example~\ref{ex:SIMONS} using separable supermodular functions.
Moreover, the four orders satisfy antisymmetry. 
They are not plagued by the same problem as $\leq_{\rm cx}$ and $\leq_{\rm cx}^{\dagger}$ that $L^0$ contains random variables for which there are not enough elements in $\mathcal{U}_{\rm cx}$ yielding finite or well-defined expectations to fully characterize the distribution. Functions $\{\mathbf {x}\mapsto\prod_{i=1}^d \id_{\{x_i\leq w_i\}}:\mathbf {w}\in\mathbb{R}^d\}$ are supermodular, bounded, right-continuous, and the sequence of their expectations fully characterizes the distributions. As a result, one cannot design counter-examples similar to the one in Example \ref{ex:different-def} where the $^\dagger$convex order relation holds simply because only constant functions were to be compared.   

As indicated in Proposition~\ref{prop:conc}, when $d\geq 3$, the concordance   and   $^\diamond$supermodular orders are different. 
The $^\diamond$supermodular order $\leq_{\rm sm}^\diamond$  is also distinct from $\leq_{\rm sm}$ and $\leq_{\rm sm}^{\dagger}$ when $d\geq 3$.
The difference can be inferred from our discussion in the Introduction and the following proposition. 

\begin{proposition}
\label{prop:diamond}
    For any $d\ge 2$ and $\mathbf{X} \in (L^0)^d$, we have
    \begin{equation}
    \label{eq:upper-diamond}
     \mathbf{X} \le_{\rm sm}^{\diamond} \mathbf{X}^{\rm co},
    \end{equation}
    where $\mathbf{X}^{\rm co}$ is a comonotonic version of $\mathbf{X}$. For $d\ge 3$, 
        \begin{equation*}
     \mathbf{X} \le_{\rm sm}^{\dagger} \mathbf{X}^{\rm co} \mbox{~and~}  \mathbf{X} \le_{\rm sm}  \mathbf{X}^{\rm co} ~ \mbox{fail to hold for some $
     \mathbf X\in (L^\infty)^d$.}
    \end{equation*}
\end{proposition}

\begin{proof}
    The first part of this result is due to \cite{L53}. A slightly more general version can also be found in \citet[Theorem~5]{tchen1980} along with a shorter proof. Example~\ref{ex:SIMONS} shows that \eqref{eq:upper-diamond} cannot be extended to $\le_{\rm sm}$ or $\le_{\rm sm}^{\dagger}$ on $(L^\infty)^d$. Note that the function $\varphi$ therein is unbounded, but it produces finite expectations for the two random random variables $\varphi(\mathbf X)$
    and $\varphi(\mathbf X^{\rm co})$
    in the example. Therefore, it is an element of $\Phi_{\rm sm}^d$ required to establish comparisons according to $\leq_{\rm sm}$ and $\leq_{\rm sm}^{\dagger}$, but not $\le_{\rm sm}^{\diamond}$. 
\end{proof}

If one reduces the domain of the supermodular order via the restriction \eqref{eq:bfrak}, then the issue raised by Example~\ref{ex:SIMONS} does not occur, as the function $\varphi$ in \eqref{eq:COUNTEREXAMPLE} does not satisfy \eqref{eq:bfrak} for any $b_{\mathfrak F}$  
with $\E[b_{\mathfrak F}(X,Y,Z)]<\infty$, and hence such $(X,Y,Z)$ is excluded. 


Proposition~\ref{prop:diamond} indicates that, in general, $\le_{\rm sm}^{\diamond} \,\,\,\,\not\!\!\!\!\implies \le_{\rm sm}$ and $\le_{\rm sm}^{\diamond} \,\,\,\,\not\!\!\!\!\implies \le_{\rm sm}^{\dagger}$ when $d\geq 3$, even on $(L^{\infty})^d$. 
In other words, $\leq_{\rm sm}$ and $\leq_{\rm sm}^{\dagger}$ cannot be generated  solely by right-continuous bounded functions. 
By Theorem 2.4.2(g) of \cite{muller2002}, this means that the supermodular order and the $^{\dagger}$supermodular order are not closed with respect to weak convergence.

In contrast to Proposition~\ref{prop:diamond}, counter-monotonic random vectors in dimensions $d\geq 3$, when they exist, are indeed the smallest elements within their Fr\'echet class with respect to $\le_{\rm sm}$ (hence also $\leq_{\rm sm}^{\dagger}$
and $\leq_{\rm sm}^{\diamond}$).  

\begin{theorem}
\label{th:counter-SM}
    Let $d\geq 3$. For any $\mathbf{X} \in (L^0)^d$, with at least $3$ non-degenerate components, we have
    \begin{equation*}
     \mathbf{X}^{\rm ct} \le_{\rm sm} \mathbf{X},
    \end{equation*}
     if a counter-monotonic version $\mathbf{X}^{\rm ct}$ of $\mathbf{X}$ exists. 
\end{theorem}

It may appear surprising that counter-monotonicity is a lower bound with respect to $\leq_{\rm sm}$, while comonotonicity fails to produce an upper bound. 
The asymmetry comes from the specific requirements on a random vector's marginal distributions for a counter-monotonic version of it to exist. For instance, a vector comprising 3 standard uniform random variables, as in Example~\ref{ex:SIMONS}, do not admit a counter-monotonic version; see \cite{D72}.  
Note that 
 $  
     \mathbf{X}^{\rm ct} \le_{\rm sm}^\diamond \mathbf{X}
     $
holds true without requiring $\mathbf X$ to have at least $3$ non-degenerate components, and this follows by combining Theorem \ref{th:counter-SM} and Proposition \ref{th:prop-sm}. 
Whether $  
     \mathbf{X}^{\rm ct} \le_{\rm sm} \mathbf{X}
     $
      or $  
     \mathbf{X}^{\rm ct} \le_{\rm sm}^{\dagger} \mathbf{X}
     $
     holds without this condition is unclear to us, and it boils down to whether the relations in Proposition \ref{th:prop-sm} hold for $\le_{\rm sm}$ and $\le_{\rm sm}^{\dagger}$.


\section{Implications for risk management and decision making}
\label{sec:r1}

As we stated in the Introduction, the convex order is a prominent tool for risk management and decision making. We hereby discuss possible insights and some limitations of our main results in these contexts.

\subsection{Choosing between two risks with possibly infinite mean}

Suppose that a decision-maker   is faced with having to choose between two infinite-mean risks. Such a situation is not unrealistic: despite the financial world containing a finite amount of money, infinite-mean distributions are relevant to multiple risk-modeling situations, catastrophic risks and operational risks for example; see \cite{CW25} for a review. If the decision-maker is risk averse, its preferences for one of the two risks should be consistent with any convex ordering of the two. This is  how  strong risk aversion is defined  for finite-mean risks, following the classical paper by \cite{RS70}. Convex order thus discerns risk in some sorts; more precisely, all risk-averse agents agree that, of two risks where one is dominated by the other, the latter is the riskier.  
A risk measure that is increasing with respect to $\leq_{\rm cx}$ is called \emph{$\le_{\rm cx}$-consistent}. Some examples are the Expected Shortfall (ES)  at level $p\in (0,1)$, $$\mathrm{ES}_p(X) = \frac{1}{1-p}\int_p^1 F_X^{-1}(u) \,\mathrm{d}u,~~~~X\in L^0,$$
whose $\le_{\rm cx}$-consistency can be checked by Theorem \ref{th:ES-based-condition},
and the entropic risk measure (ER) with parameter $\beta>0$, $$\mathrm{ER}_\beta(X) = \frac{1}{\beta}\log \mathbb{E}[ e^{\beta X}],~~~~X\in L^0,$$
whose $\le_{\rm cx}$-consistency can be checked by definition. Both ES and ER are well-defined on all of $L^0$, although they can take infinite values.

A  characterization of all monetary risk measures on $L^1$ satisfying $\le_{\rm cx}$-consistency is obtained in \cite{MW20}. 
Many $\le_{\rm cx}$-consistent risk measures are also convex, and all convex risk measures on $L^p$ for $p\ge 1$ satisfying the Fatou property are $\le_{\rm cx}$-consistent.   
Indeed, \cite{FS12}  showed that the canonical space for convex risk measures is $L^1$. General convex risk measures on $L^0$ are delicate objects to study, as extending from $L^1$ to $L^0$ may lead to undefinedness, and  there is no natural notion of continuity on $L^0$   to help establish $\le_{\rm cx}$-consistency. 
Below, we consider a popular class of risk measures, the class of distortion risk measures (e.g., \cite{WYP97}, \cite{MFE15}), which we define on $L^0$. They are defined as
\begin{equation}
    \label{eq:r1-DRM}
\rho(X) = \int_{0}^\infty h(\p(X>x)) \d x +\int_{-\infty}^0 ( h(\p(X>x)) - 1 ) \d x,~~~~X\in L^0,
\end{equation}
where $h:[0,1]\to [0,1]$ is an increasing function satisfying $h(0)=0$ and $h(1)=1$ such that \eqref{eq:r1-DRM} is well-defined for all $X\in L^0$. The function $h$ is called the distortion function of $\rho$. 
It is well known that on $L^\infty$, $\rho$ is convex if and only if $h$ is concave, which is again equivalent to $\le_{\rm cx}$-consistency;  see \citet[Theorem 3]{WWW20}. 
For $\rho$ to be well-defined on $L^0$, some condition on $h$ needs to be imposed; see \cite{WWW20a} for some examples.

\begin{theorem}\label{th:r1-DRM}
For a distortion risk measure $\rho: L^0\to [-\infty,\infty]$ with distortion function $h$, the following are equivalent:
\begin{enumerate}[label=\rm(\roman*), ref=(\roman*)]
    \item \label{item:rho-cons-0} $\rho$ is $\le_{\rm cx}$-consistent on $L^0$;
    \item \label{item:rho-cons-infty} $\rho$ is $\le_{\rm cx}$-consistent on $L^\infty$;
    \item \label{item:rho-cx} $\rho$ is convex on $L^0$;
    \item \label{item:rho-cx-infty} $\rho$ is convex on $L^\infty$;
    \item \label{item:h-concave} $h$ is concave.
\end{enumerate}
\end{theorem}

In the usual framework where the  domain of $\leq_{\rm cx}$   is limited to $L^1$, one would have to exclude infinite-mean risks from comparison, even if  ES, ER, and  some other distortion risk measures are well-defined for such risks. 
 In  Sections~\ref{sect:two-def} and~\ref{sect:conditions}, we discussed that there is no need to restrain the set of loss functions or the set of compared risks for the framework given by $\leq_{\rm cx}$ to make sense. The convex order $\leq_{\rm cx}$ remains well-behaved and, although it no longer satisfies antisymmetry, it continues to satisfy its principal properties and is well-suited to compare risks in $L^0\backslash L^1$.  Let us note that risks having infinite means do not necessarily imply infinite values of $\leq_{\rm cx}$-consistent risk measures, as we will see in Example~\ref{ex:application} below. 

One of our main results, Theorem~\ref{th:cx-cx}, pertains to risk aggregation, stipulating that comonotonicity and counter-monotonicity yield bounds on riskiness, as discerned by $\leq_{\rm cx}$, of the aggregate risk, and thus
\begin{equation}
\label{eq:r1-order}
    \rho (X^{\rm ct} + Y^{\rm ct}) \leq   \rho  (X + Y) \leq    \rho  (X^{\rm co} + Y^{\rm co}),
    ~~~\mbox{for all $\le_{\rm cx}$-consistent $\rho$},
\end{equation} 
assuming that these are well-defined.  Theorem \ref{th:r1-DRM} allows us to apply \eqref{eq:r1-order} to convex distortion risk measures on $L^0$.
The relation \eqref{eq:r1-order}  was previously inaccessible for risks outside $L^1$. The following example illustrates. 

\begin{example}
\label{ex:application}
Let $X$ and $Y$ be identically distributed random variables following a negated Pareto distribution with shape parameter 0.8 and scale parameter 1, written as $-X,-Y\sim\mathrm{Pareto}(0.8)$, that is,
$$\p(X\le  x)= \p(Y\le  x)=
(1-x)^{-0.8},~~~~x\le 0.
$$
Note that $X,Y\in L^0\backslash L^1$. 
From \eqref{eq:r1-order} and the $\le_{\rm cx}$-consistency of $\ES_p$ and $\mathrm{ER}_\beta$, we can immediately obtain  chains of inequalities for $\ES_p$ and $\mathrm{ER}_\beta$.
Numerical calculations confirm this for $p\in\{0.75,0.9\}$ and $\beta \in \{0.1,0.5\}$, as shown in Table~\ref{tab:applications}. None of these values are infinite. 
Although the patterns in Table~\ref{tab:applications}  are  intuitive, especially for $\mathrm{ES}_p$ since its right-tailedness naturally eliminates the negative-infinite mean,  
previous formulations of \eqref{eq:main} would not have allowed for such deductions since they exclude $X$ and $Y$ from comparison. 
\end{example}

 \begin{table}[tb]
 \centering
 \begin{tabular}{l cccc}
 \toprule
        & $\mathrm{ES}_{0.75}$ & $\mathrm{ES}_{0.90}$& $\mathrm{ER}_{0.1}$ & $\mathrm{ER}_{0.5}$ \\
\midrule
$X^{\rm ct} + Y^{\rm ct}$&  $-3.223 $& $ -1.278$ &  $-7.07$&   $ -4.56$\\
$X^{\perp}+Y^{\perp}$&      $-0.987 $& $-0.497$ &   $-6.02$&  $ -2.98 $ \\
$X^{\rm co}+Y^{\rm co}$&   $-0.386 $& $-0.135$ &  $-4.52$ &  $ -2.12 $ \\
 \bottomrule  
 \end{tabular}
 \caption{Risk measures $\mathrm{ES}_p$, $p\in\{0.75,0.90\}$, and $\mathrm{ER}_{\beta}$, $\beta\in\{0.1,0.5\}$, of the sum of two identically-distributed negated-Pareto risks $X,Y$ such that $-X,-Y\sim\mathrm{Pareto}(0.8)$ with various dependence relations: $(X^{\rm ct}, Y^{\rm ct})$ is a counter-monotonic version of $(X,Y)$, $(X^{\perp}, Y^{\perp})$ is an independent version, and $(X^{\rm co}, Y^{\rm co})$ is a comonotonic version.}
 \label{tab:applications}
 \end{table}



\begin{table}[tb]
 \centering
 \begin{tabular}{l |cc |cc| }

         & \multicolumn{2}{c|}{$  \mathrm{Cauchy} $} & \multicolumn{2}{c|}{$ \mathrm{Pareto}(0.8)$} \\ 
     & $\mathrm{ES}_p$ & $\mathrm{ER}_{\beta}$ & $\mathrm{ES}_p$ & $\mathrm{ER}_{\beta}$\\
 \midrule  
$X^{\rm ct} + Y^{\rm ct}$ &$0$ & $0$ & $\infty$ & $\infty$\\
$X^{\perp} + Y^{\perp}$ & $\infty$ & $\infty$ & $\infty$ & $\infty$ \\
$X^{\rm co} + Y^{\rm co}$ & $\infty$ & $\infty$ & $\infty$ & $\infty$ \\
  \bottomrule
 \end{tabular}
 \caption{Convex-order-consistent measurements of risk, although possibly  infinite, are indeed ordered according to Theorem~\ref{th:CX}. Here, the values $p\in (0,1)$ and $\beta>0$ are irrelevant.}
 \label{tab:applications-infinity}
 \end{table}

Even if risk measurements end up being infinite under some dependence schemes, they remain ordered according to \eqref{eq:r1-order}; we provide two simple examples in Table~\ref{tab:applications-infinity}. 
From Theorems~\ref{th:CXm} and~\ref{th:CXm-counter-monotonic}, similar conclusions to the ones above translate to portfolios of multiple risks. 

On the contrary, in some   settings of decision making, risks may not   aggregate linearly and each matters, by its own nature (e.g., health and wealth), giving rise to cross-preferences (see, e.g., \cite{ERS07}). In such settings, a decision-maker whose multivariate loss function is supermodular and unbounded might prefer a portfolio of comonotonic risks over some other dependence relations, as we have shown in Example~\ref{ex:SIMONS}.

\subsection{Diversification for infinite-mean risks}

Recently, the literature on risk aggregation has seen a lot of interest been addressed to diversification for risks with infinite means 
\citep{CEW25, chen2024technical, CHSZ25, chen2024super, ALO25,  M24, V25}. All of these papers discuss stochastic dominance of the form 
\begin{equation}
    (\theta_1 + \theta_2)X \leq_{\rm st} \theta_1 X + \theta_2 Y \quad \text{for}~\theta_1,\theta_2>0,
    \label{eq:STO-DOM-infty}
\end{equation}
where $X,Y$ both follow some infinite-mean distribution, 
and generalizations to $n$ random variables. In particular, \cite{chen2024technical} established that \eqref{eq:STO-DOM-infty} holds for weakly negatively-associated identically-distributed infinite-mean Pareto risks; \cite{M24}, for iid super-Cauchy risks; and \cite{ALO25}, for iid InvSub risks. 

A risk-management interpretation of \eqref{eq:STO-DOM-infty} is that diversification of the risk portfolio is harmful for these infinite-mean risks. Note that the left-hand side of \eqref{eq:STO-DOM-infty} may be rewritten as $\theta_1 X^{\rm co} + \theta_2 Y^{\rm co}$, meaning comonotonicity is the less risky dependence scheme in these cases (less than negative dependence, even, for Pareto-distributed risks, see \citet[Theorem~1]{chen2024technical}). This notion of ``less riskiness" is modelled by 
the stochastic order instead of the convex order. 

Such a statement seems to contrast with our discussion from the previous subsection. Lemma~\ref{th:lemma} part \ref{item:positive} makes clear that 
Theorem~\ref{th:CX} is actually compatible with such orderings of the form \eqref{eq:STO-DOM-infty};
 we first discuss this through the following example.

\begin{example}
	Suppose $X$ and $Y$ are iid following a Pareto distribution with parameter $\alpha\in(0,1]$, that is, $\p(X>x)=x^{-\alpha}$ for $x\in [1,\infty)$. Then $X,Y \in L^{0}\backslash L^{1}$. For any $\theta_1,\theta_2 >0$,	
Theorem~\ref{th:CX} implies
$$
		\theta_1 X +\theta_2 Y \leq_{\rm cx} (\theta_1 X)^{\rm co} + (\theta_2 Y)^{\rm co} \stackrel{\rm d}{=}  (\theta_1+\theta_2) X,
 $$
which, by  Lemma \ref{th:lemma} part \ref{item:positive}, amounts to 
   	\begin{equation}
	 		\theta_1 X +\theta_2 Y \leq_{\rm dcx} (\theta_1+\theta_2) X.
	\label{eq:mutu-cx}
	\end{equation}
 Meanwhile, Theorem~1 of \cite{chen2024technical} states that \eqref{eq:STO-DOM-infty} holds, which is clearly  stronger than \eqref{eq:mutu-cx}.
While \eqref{eq:STO-DOM-infty} holds for some (but not all) other classes of distributions without finite mean, the relation \eqref{eq:mutu-cx} holds for all distributions of $X$ with infinite $\E[X_+]$. 
 
\end{example}

As we explained in Section~\ref{sect:conditions}, the convex order can loosely be interpreted as indicating deviation from the mean. This interpretation brings clarity as to why a reverse ordering of riskiness with respect to stochastic dominance, as in \eqref{eq:STO-DOM-infty}, may hold simultaneously to \eqref{eq:main}: a distribution deviating more from the infinite mean may end up stochastically smaller. 
Thus follows a shift in our perception of risk: since the mean is infinite, we now \textit{want} to deviate from it as much as possible; deviation from the mean is no longer seen as risky. The   conclusions of the previously cited papers, that diversification may produce more risk, actually stem from this shift in perception. Random variables in $L^0$ still behave the same when summed, and comonotonicity still produces the most deviation from the mean. 

In general, $\leq_{\rm st}$ is seen as much stronger in terms of comparing risks and perhaps most will agree that risk preferences should follow $\leq_{\rm st}$, even if they do not follow $\leq_{\rm cx}$. Therefore,  $\le_{\rm cx}$-consistent preferences may not adequately capture most's conception of riskiness on $L^0$.



\subsection{Limitations}

The results of this paper were mainly developed in the perspective of better understanding convex order for infinite-mean risks; as such, their practical applications have some limitations. 

First, as argued above, 
the convex order perhaps does not encapsulate adequately what constitutes ``risk" for infinite mean random variables.  In other words, while our results provide clarifications on the diversification behavior of infinite-mean risk, reconciling the convex ordering \eqref{eq:main} with \eqref{eq:STO-DOM-infty} observed for many distributions in $L^0\backslash L^1$, they lose some part of applicability due to these very clarifications.   

Second, one could also argue about the limited practicality of convex orderings when evaluations of risk are infinite. 
As highlighted in Table~\ref{tab:applications-infinity},  ES and ER  of the sum of Pareto-distributed risks are infinite, no matter the dependence relations between these risks. Any expected losses with increasing convex loss function would be infinite under all dependence schemes as well. It may be conceptually irrelevant to ask a decision-maker to choose between risks they inherently value as infinite: if truly they assess the riskiness as infinite under each dependence relation, the decision-maker should be indifferent to it --- they are in utter despair in every case. Thus, any possible convex ordering does not matter.

Our intent was mainly to provide mathematical insights and a more in-depth understanding of convex order for risks outside of $L^1$. 
We humbly leave to more philosophically inclined minds the considerations on whether a truly infinite level of risk may exist and, if so, be manageable in any sorts.

\section{Proofs of Theorems \ref{th:ES-based-condition}--\ref{th:CX}  and \ref{th:CXm-counter-monotonic}--\ref{th:r1-DRM}
\label{sect:CXsum-proof-Theorem1}}

The proofs of Theorems \ref{th:ES-based-condition} and \ref{th:dagger-ES-based-condition} rely  on 
 Lemma~\ref{th:lemma} in Section \ref{sect:conditions}, which we prove below. 
 \begin{proof}[Proof of Lemma \ref{th:lemma}]
Note that for any $u\in \mathcal U_{\rm cx}$, if $u(y)>u(x)$ for some $y>x$, then $u$ dominates a function of the form $z\mapsto az+b$ for some $a>0$ and $b\in \R$.
    Under the assumption $\E[Y_+]=\infty$ in item~\ref{item:positive}, this $u$ produces an infinite  or undefined $\E[u(Y)]$ so that
  $
        \E[u(X)] \le \E[u(Y)]
   $
    holds necessarily whenever it is well defined. Hence, remaining to establish convex ordering is to verify the inequality for decreasing $u$, that is, $u\in\mathcal{U}_{\rm dcx}$.  To prove item~\ref{item:negative}, one puts forth the same reasoning symmetrically. For item~\ref{item:trivial},
   it follows from   items~\ref{item:positive} and \ref{item:negative}, 
   because $\mathcal{U}_{\rm dcx}\cap \mathcal{U}_{\rm icx}$ contains only constant functions. In this case, $Y\leq_{\rm cx}^{\dagger} X$ also hold because all $u\in\mathcal{U}_{\rm dcx}\cup \mathcal{U}_{\rm icx}$ are overlooked in establishing comparison as they produce infinite expectations.     
\end{proof}

\begin{proof}[Proof of Theorem~\ref{th:ES-based-condition}]

First,  it is straightforward to verify that \ref{item:cond-2} and \ref{item:cond-3} are equivalent. 
For instance, when $\E[X_+]$ and $\E[Y_+]$ are finite, the equivalence between the second condition in \eqref{eq:cond-ES} and that in \eqref{eq:stoploss} is known from the proof of \citet[Proposition~3.4.7]{DDGK05} or  \citet[Equations (4.A.4)--(4.A.8)]{shaked2007}.
When $\E[X_+]$ or $\E[Y_+]$  is infinite, it is clear that  the second condition in \eqref{eq:cond-ES} and that in \eqref{eq:stoploss} are equivalent. 

The direction \ref{item:cond-1}$\Rightarrow$\ref{item:cond-3} follows immediately, because both $x\mapsto(x-w)_+$ and $x\mapsto(x-w)_-$ are convex functions for all $w\in\mathbb{R}$. 
Next, we show \ref{item:cond-3}$\Rightarrow$\ref{item:cond-1}.   Since \eqref{eq:cond-ES} is equivalent to \eqref{eq:alter-cx} when  $X,Y\in L^1$, we only need to argue the cases when at least one of  $X,Y$ is in  $L^0\setminus L^1$. 
Clearly, if $X$ is in  $L^0\setminus L^1$, so is $Y$; we will therefore assume $Y\in L^0\setminus L^1.$
If both $\mathbb{E}[Y_+]=\infty$ and $\mathbb{E}[Y_-]=\infty$, then  $X\leq_{\rm cx}Y$ is always true by Lemma~\ref{th:lemma}, item~\ref{item:trivial}. 
Next assume $\mathbb{E}[Y_+]<\infty$ and $\mathbb{E}[Y_-]=\infty$. Note that both expectations in the second inequality of \eqref{eq:stoploss} are finite, and it is well-known that functions of the form $x\mapsto (x-w)_+$ then suffice to generate $\leq_{\rm icx}$ on $L^1$, that is, for $Z,W\in L^1$, 
\begin{equation}
    \label{eq:cond-icx}
    \E[(Z-w)_+] \le     \E[(W-w)_+] 
    \mbox{~for all $w\in \R$} 
\implies 
Z\le_{\rm icx} W;
\end{equation}
see, e.g., \citet[Theorem~1.5.7]{muller2002}. 
Note that for any $m,w\in \R$,
$$((X \vee m)-w)_+ =\begin{cases}
    (X-w)_+ & m\le w;
     \\
    (X-m)_+ + (m-w) & m>w.
\end{cases} $$
Hence, the second condition in \eqref{eq:stoploss} implies 
 $ \E[(X\vee m-w)_+]\le      \E[(Y\vee m -w)_+] $.
 Using \eqref{eq:cond-icx}, we get $(X\vee m)\le_{\rm icx}  (Y\vee m)$ for all $m\in \R$.
For any 
$f\in \mathcal U_{\rm icx}$ with $\E[f(X)]$ and $\E[f(Y)]$ well-defined,
the monotone convergence theorem guarantees that $\E[f(X \vee m)] \to \E[f(X)]$
 and $\E[f(Y \vee m)] \to \E[f(Y)]$
as $m\to-\infty$. 
Using $\E[f(X\vee m)]\le \E[f(Y\vee m)]$, we get $\E[f(X)]\le \E[f(Y)]$. 
  This further gives $X\le_{\rm cx} Y$ by using  Lemma~\ref{th:lemma}, item~\ref{item:negative}. The case where $\mathbb{E}[Y_+]=\infty$ and $\mathbb{E}[Y_-]<\infty$ is covered by a symmetric argument, involving Lemma~\ref{th:lemma}, item \ref{item:positive}. 
\end{proof}

\begin{proof}[Proof of Theorem~\ref{th:dagger-ES-based-condition}]
The equivalence between \ref{item:dagger-cond-2} and \ref{item:dagger-cond-3} is straightforward by equivalence of corresponding inequalities in \eqref{eq:cond-ES} and \eqref{eq:stoploss}. As for Theorem~\ref{th:ES-based-condition}, the direction \ref{item:dagger-cond-1}$\Rightarrow$\ref{item:dagger-cond-2} follows by noting that $x\mapsto (x-w)_+$ and $x\mapsto (x-w)_-$ are convex for any $w\in\mathbb{R}$. 
If $X,Y\in L^1$, then item \ref{item:dagger-cond-2} boils down to \eqref{eq:alter-cx}, and $\le_{\rm cx}$ and $\le_{\rm cx}^{\dagger}$ are equivalent by Proposition~\ref{prop:equivalent-def}. 

We argue \ref{item:cond-3}$\Rightarrow$\ref{item:dagger-cond-1} for cases when at least one of $X,Y$ has an infinite mean. 
Suppose that $\mathbb{E}[X_-]$ is infinite. Then all $u\in\mathcal{U}_{\rm dcx}$, other than the constant functions, will produce an infinite $\mathbb{E}[u(X)]$. Hence, all non-constant $u\in\mathcal{U}_{\rm dcx}$ are   excluded when checking $X\leq_{\rm cx}^{\dagger} Y$ according to the definition of $\leq_{\rm cx}^{\dagger}$. We therefore only need to assess $\mathbb{E}[u(X)]\leq\mathbb{E}[u(Y)]$ for $u\in\mathcal{U}_{\rm icx}$; the second inequality in \eqref{eq:stoploss} suffices as discussed in the proof of Theorem~\ref{th:ES-based-condition}. The same argument also applies if rather $\mathbb{E}[Y_-]$ is infinite, and it is symmetric if $\mathbb{E}[X_+]=\infty$ or $\mathbb{E}[Y_+]=\infty$. If $\mathbb{E}[X]$ is not well-defined, then neither inequality in \eqref{eq:stoploss} needs to the checked, and $X\leq_{\rm cx}^{\dagger}Y$ by Lemma~\ref{th:lemma}, item \ref{item:cond-3}.
\end{proof}

The proof of Theorem \ref{th:CX} relies on extending the desired relations \eqref{eq:main}  from  $L^{\infty}$, as stated in the following lemma, to $L^0$. 


\begin{lemma}[Corollary 3.28 of \cite{R13}, restated on $L^\infty$]\label{lem:thL1}
For $X,Y\in L^\infty$, \eqref{eq:main} holds. 
\end{lemma}



\begin{proof}[Proof of Theorem \ref{th:CX}.]  The main idea of the proof is that  appropriately truncating the random variables allows to apply  Lemma \ref{lem:thL1} on $L^{\infty}$, and then to exploit monotone convergence to extend it to the original random variables. We present the developments for the lower bound in \eqref{eq:main}. By changing ``ct'' to ``co'' and flipping the inequalities, one will immediately obtain the proof for the upper bound.

For $m>0$ and a random variable $Z$, let $Z_{(m)}=\max\{\min\{Z,m\},-m\}$, that is, $Z$ truncated at $\pm m$.
The   terms  $X_{(m)}$, $Y_{(m)}$, $X^{\rm ct}_{(m)}$ and $Y^{\rm ct}_{(m)}$ are truncations in the above sense.
Note that $ X^{\rm ct}_{(m)}$ and $Y^{\rm ct}_{(m)} $ are counter-monotonic since truncation is a monotone transform. 
Moreover, $X_{(m)} \laweq X^{\rm ct}_{(m)}$  and  $Y_{(m)} \laweq Y^{\rm ct}_{(m)}$.
Therefore,  Lemma \ref{lem:thL1} implies \begin{align}\label{eq:L1ct}
\E[u(X^{\rm ct}_{(m)}+Y^{\rm ct}_{(m)})]
\le \E[u(X_{(m)}+Y_{(m)})] \mbox{~~~for all $u\in\mathcal{U}_{\rm cx}$},
\end{align}   as soon as the two expectations are well-defined. 

Let $\mathcal U_{\rm cx}^*$ be the set of all $u\in\mathcal{U}_{\rm cx}$ such that $u$ is monotone (i.e., either increasing or decreasing) on $[0,\infty)$ and on $(-\infty,0]$. 
Note that any convex function $u\in \mathcal U_{\rm cx}$  can change monotonicity at most once, say at $x_0$.  Then, the function $v:x\mapsto u(x-x_0)$ is in $\mathcal U_{\rm cx}^*$. 
If $u$ does not change monotonicity, then $u\in \mathcal U_{\rm cx}^*$. 

Next, we make the simple, but nontrivial, observation that $(X_{(m)}+Y_{(m)})_+$ and $(X_{(m)}+Y_{(m)})_-$ are increasing in $m$.  To see this, take  $(X_{(m)}+Y_{(m)})_+$ for instance. The lower truncation at $-m$ does not matter for the sum, 
and one has  $$(X_{(m)}+Y_{(m)})_+=  (\min\{X,m\}+\min\{Y,m\})_+ .$$ 
This observation was made by \cite{S77}.

Denote by $A_{m}=\{ X_{(m)}+Y_{(m)} \ge 0\}$ and $A_m^c$ its complement.
Take any $u\in \mathcal U_{\rm cx}^*$ such that  $\E[u(X+Y)]$ is well-defined.
We have
\begin{align*}
\E[u(X _{(m)}+Y _{(m)})]
&= \E[u( X _{(m)}+Y _{(m)}  ) \id_{A_m} + u( X _{(m)}+Y _{(m)}  )  \id_{A^c_m}   ]
\\ &= \E[u((X _{(m)}+Y _{(m)})_+) \id_{A_m} + u(-(X _{(m)}+Y _{(m)})_-)\id_{A^c_m}    ]
\\ &= \E[u((X _{(m)}+Y _{(m)})_+)-u(0) \id_{A^c_m} + u(-(X _{(m)}+Y _{(m)})_-)-u(0) \id_{A_m}    ]
\\ &= \E[u((X _{(m)}+Y _{(m)})_+)  + u(-(X _{(m)}+Y _{(m)})_-)  ]-u(0).
\end{align*}
Similarly, 
$$
\E[u(X +Y )] = \E[u((X +Y )_+)  + u(-(X +Y )_-)  ]-u(0). 
$$
Because $(X_{(m)}+Y_{(m)})_+\uparrow (X+Y)_+$
and $(X_{(m)}+Y_{(m)})_-\uparrow (X+Y)_-$  as $m\to\infty$,  
the monotone convergence theorem  yields
$\E[u(X _{(m)}+Y _{(m)})] \to  \E[u(X +Y )] $ as $m\to\infty$. 
Applying this to $(X^{\rm ct},Y^{\rm ct})$, we get 
$\E[u(X^{\rm ct}_{(m)}+Y^{\rm ct}_{(m)})] \to  \E[u(X^{\rm ct}+Y^{\rm ct})] $ as $m\to\infty$ if $\E[u(X^{\rm ct}+Y^{\rm ct})]$   is well-defined.
Using \eqref{eq:L1ct}, we obtain $\E[u(X^{\rm ct}+Y^{\rm ct})]\le \E[u(X+Y)] $.

To summarize, we have shown \begin{align}\label{eq:L0ct}
\E[u(X^{\rm ct}+Y^{\rm ct})]\le \E[u(X+Y)]  \mbox{~~~~for all $X,Y\in L^0$ and all $u\in \mathcal U_{\rm cx}^*$},
\end{align} as soon as the expectations are well-defined.

For general $u\in \mathcal U_{\rm cx}$ but not in $\mathcal U_{\rm cx}^*$, we can find $x_0\in \R$ and $v:x\mapsto u(x-x_0)$ such that $v\in \mathcal U_{\rm cx}^*$, which we argued above.
For $X,Y\in L^0$ with well-defined expectations $\E[u(X+Y)]$ and $\E[u(X^{\rm ct}+Y^{\rm ct})]$, take $Z=Y + x_0$. Note that $Z^{\rm ct} =Y^{\rm ct} + x_0$. Using  the result \eqref{eq:L0ct} applied to $(v,X,Z)$, we get 
\begin{align*} 
\E[u(X^{\rm ct}+Y^{\rm ct})]& =\E[u(X^{\rm ct}+Z^{\rm ct} - x_0)]
\\&=\E[v(X^{\rm ct}+Z^{\rm ct}  )]
\\& \le \E[v(X +Z  )] 
= \E[u (X +Z-x_0  )]  =\E[u(X+Y)]   .
\end{align*}
Therefore, we obtain the desired result that \eqref{eq:L0ct} holds for all $u\in \mathcal U_{\rm cx}$,
and thus $X^{\rm ct}+Y^{\rm ct}\le_{\rm cx} X+Y$.
\end{proof}

 Note that 
the proof of Theorem~\ref{th:CX} does not generalize to the multivariate setting, so Theorems \ref{th:CXm} and \ref{th:CXm-counter-monotonic} require separate proofs. The proof of Theorem \ref{th:CXm} was presented in Section~\ref{sect:CXsum-explanations}.

\begin{proof}[Proof of Theorem~\ref{th:CXm-counter-monotonic}]
    The case $d=1$ is trivial and $d=2$ is covered by Theorem~\ref{th:CX}. The cases where at most $2$ random variables are non-degenerate are then also covered because constant shifts do not affect the comparison $\mathbb{E}[u(\sum_{i=1}^d X^{\rm ct}_i)]\leq \mathbb{E}[u(\sum_{i=1}^d X_i)]$, $u\in\mathcal{U}_{\rm cx}$. For all other cases where $d\geq 3$ and a counter-monotonic version of $(X_1,\ldots,X_n)$ exists, the result ensues directly from Theorem~\ref{th:counter-SM}, whose proof follows below, and the supermodularity of functions of the form $(x,y)\mapsto u(x+y)$, $u\in\mathcal{U}_{\rm cx}$. 
\end{proof}

\begin{proof}[Proof of Theorem~\ref{th:counter-SM}]
By Theorem~3 of \cite{D72}, for $d\ge 3$, a $d$-dimensional random vector with at least $3 $ non-degenerate components   has a counter-monotonic version if and only if 
\begin{align}
\label{eq:cond-ct-B}
\mbox{either~~~~} & \sum_{i=1}^d \mathbb{P}(X_i > \essinf X_i) \leq 1,\\
\mbox{or~~~~} &  
\label{eq:cond-ct-A}
\sum_{i=1}^d \mathbb{P}(X_i < \esssup X_i) \leq 1 .
\end{align}
    We assume \eqref{eq:cond-ct-B}, and the other case \eqref{eq:cond-ct-A} is symmetric and omitted. 
 Note that \eqref{eq:cond-ct-B} implies that $\essinf X_i$ is finite for each $i\in [d]$, and without loss of generality we can assume $\essinf X_i=0$ for $i\in [d]$  because constant shifts do not affect the comparison $\mathbb{E}[\varphi(\mathbf{X}^{\rm ct})]\leq \mathbb{E}[\varphi(\mathbf{X})]$, $\varphi\in\Phi^d_{\rm sm}$.

Take $\varphi\in\Phi_{\rm sm}^d$ such that both $ \E[\varphi(\mathbf{X}^{\rm ct})]$  and  $ \E[\varphi(\mathbf{X}^{\rm ct})]$ are well-defined, and without loss of generality assume $\varphi(0,\dots,0) = 0$.
It is straightforward to verify that the supermodular function $\varphi$ with $\varphi(0,\dots,0)=0$ satisfies 
$$
\varphi(x_1,\dots,x_d) \ge \sum_{i=1}^d \varphi(x_i \mathbf e_i),
$$
where $\mathbf e_i=(0,\dots,0,1,0,\dots,0)$ with $1$ at the $i$th position. 
Note that $
\varphi(X_i\mathbf e_i)$
is identically distributed 
as
$\varphi(X_i^{\rm ct}\mathbf e_i)$ for each $i\in [d]$. 
Therefore, 
$\E[ \varphi(X_i\mathbf e_i)]= 
\E[ \varphi(X_i^{\rm ct}\mathbf e_i)]
$
as long as they are well-defined. 
Moreover, 
$$
  \sum_{i=1}^d  \varphi(X_i^{\rm ct}\mathbf e_i)
  = \varphi(\mathbf{X}^{\rm ct}) \mbox{~~~~almost surely,}
  $$
  because $\mathbf{X}^{\rm ct}$ is counter-monotonic; see \citet[Lemma~2]{D72} or \citet[Proposition~3.2]{PW15}. 
Using this, for $\E[\varphi(\mathbf{X}^{\rm ct})]$ to be well-defined, it holds that  $\E[ \varphi(X_i^{\rm ct}\mathbf e_i)]<\infty$ for all $i\in[d]$ or $\E[ \varphi(X_i^{\rm ct}\mathbf e_i)]>-\infty$ for all $i\in[d]$.
Therefore, we get
$$
\E[\varphi(\mathbf X)]
\ge  \sum_{i=1}^d 
 \E\left[\varphi(X_i\mathbf e_i)\right]
=   \sum_{i=1}^d 
 \E\left[\varphi(X_i^{\rm ct}\mathbf e_i)\right]
 = 
 \E\left[  \sum_{i=1}^d  \varphi(X_i^{\rm ct}\mathbf e_i)\right] 
=\E[\varphi(\mathbf{X}^{\rm ct})],
$$ 
where we used the verified fact that all expectations are well-defined. 
\end{proof}

\begin{proof}[Proof of Theorem \ref{th:r1-DRM}]
The chains \ref{item:rho-cons-0}$\Rightarrow$\ref{item:rho-cons-infty}$\Rightarrow$\ref{item:h-concave} and 
\ref{item:rho-cx}$\Rightarrow$\ref{item:rho-cx-infty}$\Rightarrow$\ref{item:h-concave} 
follow   from the definition and \citet[Theorem 3]{WWW20}. It remains to show \ref{item:h-concave}$\Rightarrow$\ref{item:rho-cons-0}
and \ref{item:h-concave}$\Rightarrow$\ref{item:rho-cx}.
Take a concave $h$ and a corresponding distortion risk measure $\rho$ in \eqref{eq:r1-DRM}. Since $h$ is increasing and concave, it is left-continuous. Using Lemma 1 of \cite{WWW20a}, we have 
$$
\rho(X) = \int_0^1 F^{-1}_X(1-t) \d h(t) ,~~~ X \in L^0. 
$$
Let $g(t)=1-h(1-t)$ for $t\in [0,1]$, with its right derivative denoted by $g' $; since $g$ is convex, $g'$ is increasing. The right derivative of $g'$ is given by $g''$, which exists almost everywhere and it is positive. 
Further, let  $F_X^{-1}(1)=\inf\{x \in \R: \p(X\le x)= 1\}$,
and $\delta=g(1)-\lim_{t\uparrow 1}g(t)$.   
It follows that  
    \begin{align*}
\rho(X) &= \int_0^1 F^{-1}_X(t) \d g(t) \\
 & = \int_0^1 F^{-1}_X(t) g'(t) \d t  + \delta F_X^{-1}(1)  \\
 & = \int_0^1 \int_0^t  F^{-1}_X(t) g''(s)  \d s \d t  + \delta F_X^{-1}(1) \\
 & = \int_0^1 g''(s)   \int_s^1  F^{-1}_X(t)  \d t  \d s   + \delta F_X^{-1}(1) .
    \end{align*}
    Since   $X\mapsto \int_s^1  F^{-1}_X(t)  \d t  $
   is $\le_{\rm cx}$-consistent  on $L^0$ 
for $s\in (0,1)$ by using Theorem \ref{th:ES-based-condition}-\ref{item:cond-2}
      and $X\mapsto F_X^{-1}(1)   $
        is $\le_{\rm cx}$-consistent on $L^0$ by using Theorem \ref{th:ES-based-condition}-\ref{item:cond-3}, we conclude that $\rho$ is $\le_{\rm cx} $-consistent on $L^0$, and thus \ref{item:rho-cons-0} holds. 
        Similarly, both $X\mapsto \int_s^1  F^{-1}_X(t)  \d t  $ and $X\mapsto F_X^{-1}(1) $  are convex due to the convexity of ES on $L^0$, and hence \ref{item:rho-cx} holds. 
\end{proof}

\subsection*{Acknowledgments}
The authors thank an editor, two anonymous referees, Alfred M\"uller, and Giovanni Puccetti for helpful comments on an early version of the paper. Ruodu Wang is supported by the Natural Sciences and Engineering Research Council of Canada (RGPIN-2018-03823, RGPAS-2018-522590) and Canada Research Chairs (CRC-2022-00141).

 \end{document}